\documentclass{amsart}
\usepackage{amsmath}
\usepackage{amssymb}
\usepackage{amsthm}
\usepackage{tikz}
\usepackage{hyperref}
\hypersetup{colorlinks=false}

\theoremstyle{plain}

\newtheorem{theorem}{Theorem}[section]
\newtheorem{lemma}[theorem]{Lemma}
\newtheorem{proposition}[theorem]{Proposition}
\newtheorem{corollary}[theorem]{Corollary}

\theoremstyle{definition}
\newtheorem{definition}[theorem]{Definition}
\newtheorem{remark}[theorem]{Remark}
\newtheorem{example}[theorem]{Example}

\newtheorem{theoremx}{Theorem}
\numberwithin{equation}{section}
\newcommand\fantome[1]{}

\DeclareMathOperator{\ev}{ev}

\DeclareMathOperator{\sgn}{sgn}

\begin{document}

\title{Stark units in positive characteristic}

\author{Bruno Angl\`es  \and Tuan Ngo Dac \and Floric Tavares Ribeiro}

\address{
Normandie Universit\'e, Universit\'e de Caen Normandie,
Laboratoire de Math\'ematiques Nicolas Oresme,
CNRS UMR 6139,
Campus II, Boulevard Mar\'echal Juin,
B.P. 5186,
14032 Caen Cedex, France.
}
\email{bruno.angles@unicaen.fr,tuan.ngodac@unicaen.fr,floric.tavares-ribeiro@unicaen.fr}

%%%%%%%%%%%%%%%%%%%%%%%%%%%%%%%%%%%%%%%%%%%%%%%%%%%%%%%%%%%%%%%%%%%%%%%%%%%

\begin{abstract}
We show that the module of Stark units associated to a sign-normalized rank one Drinfeld module can be obtained from Anderson's equivariant $A$-harmonic series. We apply this to obtain a class formula \`a la Taelman and to prove a several variable log-algebraicity theorem, generalizing Anderson's log-algebraicity theorem. We also give another proof of Anderson's log-algebraicity theorem using shtukas and obtain various results concerning the module of Stark units for Drinfeld modules of arbitrary rank.
\end{abstract}

\subjclass[2010]{Primary 11G09; Secondary 11M38, 11R58}

\keywords{Stark units, $L$-series in positive characteristics, class formula, log-algebraicity theorem, shtukas}

\date{  \today}

\maketitle

\footnotetext[1]{The second author was partially supported by ANR Grant PerCoLaTor ANR-14-CE25-0002.}

\tableofcontents

%%%%%%%%%%%%%%%%%%%%%%%%%%%%%%%%%%%%%%%%%%%%%%%%%%%%%%%%%%%%%%%%%%%%%%%%%%%

\section*{Introduction}

The power-series $\sum_{n \geq 1} \frac{z^n}{n}$ is log-algebraic:
 $$\sum_{n \geq 1} \frac{z^n}{n}=-\log(1-z). $$
This identity allows one to obtain the value of a Dirichlet $L$-series at $s=1$ as an algebraic linear combination of logarithms of circular units. Inspired by examples of Carlitz \cite{CAR} and Thakur \cite{THA3}, Anderson (\cite{AND}, \cite{AND2}) proved an analogue of this identity for a sign-normalized rank one Drinfeld $A$-module, known as Anderson's log-algebraicity theorem.

When $A=\mathbb F_q[\theta]$ (the genus 0 case), various works (\cite{AND, AND2, AND&THA, APT, APTR, ATR, DEM, FAN1, FAN2, FAN0, PAP, PEL, TAE1, TAE2}) have revealed the importance of certain units in the study of special values of the Goss $L$-functions at $s=1$. To give a simple example, the Carlitz module is considered to play the role of the multiplicative group $\mathbb G_m$ over $\mathbb Z$, and Anderson (\cite{AND, AND2}) showed that the images through the Carlitz exponential of some special units give algebraic elements which are the equivalent of the circular units. The special units constructed in such a way are then ``log-algebraic''. Recently, Taelman (\cite{TAE1, TAE2}) introduced the module of units attached to any Drinfeld module and proved a class formula which states that the special value of the Goss $L$-function attached to a Drinfeld module at $s=1$ is the product of a regulator term arising from the module of units and an algebraic term arising from a certain class module. Also, deformations of Goss $L$-series values in Tate algebras are investigated by Pellarin and two of the authors (\cite{APT, APTR, ATR, PEL}). For higher dimensional versions of Drinfeld modules, we refer the reader to \cite{AND&THA, DEM, FAN1, FAN2, FAN0, PAP}.  We should mention that all these works are based on a crucial fact that $\mathbb F_q[\theta]$ is a principal ideal domain, which is no longer true in general.

In the present paper, we develop a new method to deal with higher genus cases. We introduce Stark units attached to Drinfeld $A$-modules extending the previous work of two of the authors (\cite{ATR}) and make a systematic study of these modules of Stark units. For a sign-normalized rank one Drinfeld module, we prove a direct link between the module of Stark units and Anderson's equivariant $A$-harmonic series, which is an analogue of Stark's conjectures. It allows us to obtain a class formula \`a la Taelman and a several variable log-algebraicity theorem in the general context.

\bigskip

Let us give now more precise statements of our results.

Let $K/\mathbb F_q$ be a global function field ($\mathbb F_q$ is algebraically closed in $K$),  let $A$ be the ring of elements of $K$ which are regular outside a fixed place $\infty$ of $K$ of degree $d_\infty\geq 1.$  The completion $K_\infty$ of $K$ at the place $\infty$ has residue field $\mathbb F_\infty$ and is endowed with the  $\infty$-adic valuation $v_\infty : K_\infty \twoheadrightarrow \mathbb Z\cup \{+\infty\}$. For $a \in A,$ we set: $\deg a:=-d_\infty v_\infty(a).$ We fix an algebraic closure $\overline{K}_\infty$ of $K_\infty,$ and still denote $v_\infty: \overline{K}_\infty \twoheadrightarrow \mathbb Q\cup \{+\infty\}$ the extension of $v_\infty$ to $\overline{K}_\infty.$ Let $\tau: \overline{K}_\infty \rightarrow \overline{K}_\infty$ be the $\mathbb F_q$-algebra homomorphism which sends $x$ to $x^q.$

\medskip

We choose a sign function $\sgn:K_\infty^\times \rightarrow \mathbb F_\infty^\times$, that is, a group homomorphism such that $\sgn\mid_{\mathbb F_\infty^\times}= {\rm Id}_{\mathbb F_\infty^\times}.$ Let $\phi:A\hookrightarrow \overline{K}_\infty\{\tau\}$ be a sign-normalized rank one Drinfeld module (see Section \ref{sgn}), i.e. there exists an integer $i(\phi)\in \mathbb N$ such that:
 $$\forall a\in A, \quad \phi_a=a+\cdots +\sgn(a) ^{q^{i(\phi)}}\tau^{\deg a}.$$
Then, the exponential series attached to $\phi$ is the unique element $\exp_\phi \in \overline{K}_\infty\{\{\tau\}\},$ such that $\exp_\phi \equiv 1\pmod{\tau},$ and:
 $$\forall a\in A, \quad \exp_\phi a =\phi_a\exp_\phi.$$
If we write:
 $$\exp_\phi=\sum_{i\geq 0}e_i(\phi) \tau ^i, $$
with $e_i(\phi)\in \overline{K}_\infty,$  then the field $H:=K(e_i(\phi), i\in \mathbb N)$ is a finite abelian extension of $K$ which is unramified outside $\infty$ (see Section \ref{sgn}). Let $B$ be the integral closure of $A$ in $H.$ For all $a\in A,$ we have:
 $$\phi_a\in B\{\tau\}.$$
For a non-zero ideal $I$ of $A,$ we define $\phi_I\in H\{\tau\}$ to be the monic element in $H\{\tau\}$ such that:
 $$H\{\tau \}\phi_I=\sum_{a\in I}H\{\tau \} \phi_a.$$
In fact, $\phi_I\in B\{\tau\}$ and we denote its constant term by $\psi(I)\in B\setminus\{0\}.$

\medskip

For simplicity, we will work over the abelian extension $H/K.$ We should mention that the results presented below are still valid for any finite abelian extension $E/K$ such that $H \subset E.$

\medskip

Let $G={\rm Gal}(H/K).$ For a non-zero ideal $I$ of $A$, we denote by $\sigma_I =(I,H/K) \in G,$ where $(\cdot,H/K)$ is the Artin map. Let $z$ be an indeterminate over $K_\infty$ and let $\mathbb T_z(K_\infty)$ be the Tate algebra in the variable $z$ with coefficients in $K_\infty.$ Let's set:
$$H_\infty=H\otimes_KK_\infty,$$
and:
$$\mathbb T_z(H_\infty)=H\otimes_K\mathbb T_z(K_\infty).$$
Let $\tau:\mathbb T_z(H_\infty)\rightarrow \mathbb T_z(H_\infty)$ be the continuous $\mathbb F_q[z]$-algebra homomorphism such that:
 $$\forall x\in H_\infty, \quad \tau(x)=x^q.$$
We set:
 $$\exp_{\widetilde{\phi}}=\sum_{i\geq 0}e_i(\phi) z^i\tau^i\in H[z]\{\{\tau\}\}.$$
Then $\exp_{\widetilde{\phi}}$ converges on $\mathbb T_z(H_\infty).$ Following \cite{ATR}, we introduce the module of $z$-units attached to $\phi/B$:
 $$U(\widetilde{\phi}/B[z])=\{ f\in \mathbb T_z(H_\infty), \exp_{\widetilde{\phi}}(f)\in B[z]\}.$$
We denote by $\ev:\mathbb T_z(H_\infty)\rightarrow H_\infty$ the evaluation at $z=1.$ The module of Stark units attached to $\phi/B$ is defined by (see \cite{ATR}, Section 2):
 $$U_{St}(\phi/B)=\ev(U(\widetilde{\phi}/B[z]))\subset H_\infty.$$
Then $U_{St}(\phi/B)$ is an $A$-lattice in $H_\infty$ (see Theorem \ref{TheoremS4.1}), i.e. $U_{St}(\phi/B)$ is  an $A$-module which  is discrete and cocompact in $H_\infty.$ In fact, $U_{St}(\phi/B)$ is contained in the $A$-module of the Taelman module of units \cite{TAE1} defined by:
 $$U(\phi/B)=\{ x\in H_\infty, \exp_\phi(x)\in B\},$$
which is also an $A$-lattice in $H_\infty.$ Following Taelman \cite{TAE1}, the Taelman class module $H(\phi/B)$ is a finite $A$-module (via $\phi$) defined by:
 $$H(\phi/B)=\frac{H_\infty}{B+\exp_\phi(H_\infty)}.$$

\medskip

Following Anderson \cite{AND}, we introduce the following series (see Section \ref{Example}):
$$\mathcal L(\phi/B; 1;z)= \sum_I \frac{z^{\deg I}}{\psi(I)}\sigma_I \in \mathbb T_z(H_\infty)[G],$$
where the sum runs through the non-zero ideals $I$ of $A.$  The equivariant $A$-harmonic series attached to $\phi/B$ is defined by:
$$\mathcal L(\phi/B)=\ev(\mathcal L(\phi/B; 1;z))\in H_\infty[G].$$

One of our main theorems states (see Theorem \ref{TheoremS6.2}) that the module of Stark units $U_{St}(\phi/B)$ can be obtained from the equivariant $A$-harmonic series $\mathcal L(\phi/B)$, which is reminiscent of Stark's Conjectures (\cite{TAT2}):
\begin{theoremx} \label{key result}
We have:
 $$U(\widetilde{\phi}/B[z])= \mathcal L(\phi/B; 1;z)B[z].$$
In particular,
 $$U_{St}(\phi/B)= \mathcal L(\phi/B)B.$$
\end{theoremx}

We will present several applications of this theorem.

\medskip

Firstly, we apply Theorem \ref{key result} to obtain a class formula \`a la Taelman for $\phi/B,$ by a different method of Taelman's original one \cite{TAE2}. Roughly speaking, we introduce the Stark regulator (resp. the regulator defined by Taelman \cite{TAE2}) attached to $\phi/B$ by $[B:U_{St}(\phi/B)]_A\in \overline{K}_\infty^\times$ (resp. $[B:U(\phi/B)]_A\in \overline{K}_\infty^\times$) (see Section \ref{Prel}). We show (see Theorem \ref{TheoremS4.1}):
\begin{theoremx} \label{regulator}
We have:
 $${\rm Fitt}_A\frac{U(\phi/B)}{U_{St}(\phi/B)}= {\rm Fitt}_AH(\phi/B),$$
where, for a finite $A$-module $M,$ ${\rm Fitt}_A M$ is the Fitting ideal of $M$.
\end{theoremx}

Observe that $\mathcal L(\phi/B)$ induces a $K_\infty$-linear map on $H_\infty,$ and we denote by $\det_{K_\infty} \mathcal L(\phi/B)$ its determinant. We prove the following formula (see Theorem \ref{TheoremS6.1}):
 $${\det}_{K_\infty} \mathcal L(\phi/B)=\zeta_B(1):=\prod_{\frak P}(1-\frac{1}{[\frac{B}{\frak P}]_A})^{-1}\in \overline{K}_\infty^\times,$$
where $\frak P$ runs through the maximal ideals of $B.$ Note that $\zeta_B(1)$ is a special value at $s=1$ of some zeta function $\zeta_B(s)$ introduced by Goss (see \cite{GOS}, Chapter 8). Therefore, Theorem \ref{key result} and Theorem \ref{regulator} imply Taelman's class formula for $\phi/B$ (see Theorem \ref{TheoremS6.3}):
\begin{theoremx} \label{intro class formula}
We have:
 $$\zeta_B(1)=[B:U_{St}(\phi/B)]_A=[B:U(\phi/B)]_A[H(\phi/B)]_A.$$
\end{theoremx}

When the genus of $K$ is zero and $d_\infty=1$, Taelman's class formula, its higher dimensional versions,  and its arithmetic consequences are now well-understood  due to the recent works (\cite{APT}, \cite{AT}, \cite{ATR}, \cite{DEM}, \cite{FAN1}, \cite{FAN2}, \cite{FAN0}, \cite{TAE1}, \cite{TAE2}). All these works are based on the crucial fact that when $g=0$ and $d_\infty=1$, the ring $A$ is a principal ideal domain (when $A$ is not assumed to be principal, the existence of a class formula is still an open problem in general). Using the module of Stark units, we are able to overcome this difficulty, and Theorem \ref{intro class formula} provides a large class of  examples of Taelman's class formula when $A$ is no longer principal. We refer the reader to Section \ref{RegS} for a more detailed discussion.

\medskip

Secondly, we apply Theorem \ref{key result} to prove a several variable log-algebraicity theorem, generalizing Anderson's log-algebraicity theorems (see Theorem \ref{TheoremS6.4}). (The theorem below is valid for any finite abelian extension $E/K$, $H\subset E,$ see Theorem \ref{TheoremS6.4} for the precise statement).

\begin{theoremx} \label{log algebraicity}
Let $n\geq 0$ and let $X_1, \ldots, X_n,z$ be $n+1$ indeterminates over $K.$ Let $\tau: K[X_1, \ldots, X_n][[z]]\rightarrow K[X_1, \ldots, X_n][[z]]$ be the continuous $\mathbb F_q[[z]]$-algebra homomorphism for the $z$-adic topology such that $\forall x\in K[X_1, \ldots, X_n], \tau (x)=x^q.$ Then:
 $$\forall b\in B, \quad \exp_{\widetilde{\phi}}(\sum_I \frac{\sigma_I(b)}{\psi(I)} \phi_I(X_1)\cdots \phi_I(X_n) z^{\deg I})\in B[X_1, \ldots, X_n, z],$$
where $I$ runs through the non-zero ideals of $A.$
\end{theoremx}

For $n \leq 1$ and $d_\infty=1,$ this theorem was due to G. Anderson (\cite{AND}, Theorem 5.1.1 and \cite{AND2}, Theorem 3):
 $$\forall b\in B, \quad \exp_{\widetilde{\phi}}(\sum_I \frac{\sigma_I(b)}{\psi(I)} z^{\deg I})\in B[z],$$
 $$\forall b\in B, \quad \exp_{\widetilde{\phi}}(\sum_I \frac{\sigma_I(b)}{\psi(I)} \phi_I(X) z^{\deg I})\in B[X, z],$$
where the sum runs through the non-zero ideals of $A.$ Again, this result is now well-understood  when the genus of $K$ is zero (and $d_\infty=1$) due to the recent works of many people (\cite{APT}, \cite{APTR}, \cite{AT}, \cite{ATR}, \cite{TAE2}, \cite{THA} Sections  8.9 and 8.10, and the forthcoming work of M. Papanikolas \cite{PAP}). However, to our knowledge, Anderson's log-algebraicity remains quite mysterious for $g>0$ until now.

\medskip

Thirdly, we present an alternative approach to the previous several variable log-algebraicity theorem (Theorem \ref{log algebraicity}) via Drinfeld's correspondence between Drinfeld modules and shtukas. Using the shtuka function attached to $\phi/B$ via Drinfeld's correspondence, we introduce one variable versions of the previous objects, i.e. the modules of $z$-units and Stark units, the equivariant $A$-harmonic series and the $L$-series (see Section \ref{Section one variable}). We prove an analogue of Theorem \ref{key result} in this one variable context (see Theorem \ref{TheoremS7.1}). More generally, we also obtain a several variable log-algebraicity theorem (see Corollary \ref{CorollaryS7.1}). In the case $g=0$ and $d_\infty=1$, we rediscover the Pellarin's $L$-series \cite{PEL} and its several variable variants studied in \cite{ANDTR2}, \cite{APT},  \cite{APTR}, \cite{ATR}. We deduce from this another proof of Theorem \ref{log algebraicity} (see Section \ref{Section another proof}).

\medskip

Finally, we prove some results concerning the module of Stark units for Drinfeld modules of arbitrary rank in Section \ref{StarkL}. In particular, Theorem \ref{regulator} is still valid for any Drinfeld module.

%%%%%%%%%%%%%%%%%%%%%%%%%%%%%%%%%%%%%%%%

\section{Notation} \label{Notation}

Let $K/\mathbb F_q$ be a global function field of genus $g$, where $\mathbb F_q$ is a finite field  of characteristic $p$, having $q$ elements ($\mathbb F_q$ is algebraically closed in $K$). We fix a place $\infty$ of $K$ of degree $d_\infty,$ and denote by $A$ the ring of elements of $K$ which are regular outside of $\infty.$  The  completion $K_\infty$ of $K$ at the place $\infty$ has residue field $\mathbb F_\infty$ and comes with the $\infty$-adic valuation $v_\infty : K_\infty \twoheadrightarrow \mathbb Z\cup \{+\infty\}$. We fix an algebraic closure $\overline{K}_\infty$ of $K_\infty$ and still denote by $v_\infty: \mathbb C_\infty \twoheadrightarrow \mathbb Q\cup \{+\infty\}$ the extension of $v_\infty$ to  the completion $\mathbb C_\infty$ of $\overline{K}_\infty$.

\medskip

We will fix a uniformizer $\pi$ of $K_\infty.$ Set $\pi_1=\pi,$ and for $n\geq 2,$ choose $\pi_n\in \overline{K}_\infty^\times$ such that $\pi_n^n=\pi_{n-1}.$ If $z\in \mathbb Q,$ $z=\frac{m}{n!}$ for some $m\in \mathbb Z, n\geq 1,$ we set:
 $$\pi^z := \pi_n^m.$$
Let $\overline{\mathbb F}_q$ be the algebraic closure of $\mathbb F_q$ in $\overline{K}_\infty,$ and let $U_\infty=\{ x\in \overline{K}_\infty, v_\infty(x-1)>0\}.$ Then:
$$\overline{K}_\infty^\times= \pi^{\mathbb Q}\times \overline{\mathbb F}_q^\times\times U_\infty.$$
Therefore, if $x\in \overline{K}_\infty^\times,$ one can write in a unique way:
 $$x=\pi^{v_\infty(x)} \sgn (x) \langle x \rangle, \quad \sgn(x)\in \overline{\mathbb F}_q^\times, \langle x \rangle \in U_\infty.$$

\medskip

Let $\mathcal I(A)$ be the group of non-zero fractional ideals of $A.$ For $I\in \mathcal I(A), I\subset A,$ we set:
 $$\deg I :=\dim_{\mathbb F_q} A/I.$$
Then, the function $\deg$ on non-zero ideals of $A$ extends  into a group homomorphism:
$$\deg: \mathcal I(A)\twoheadrightarrow \mathbb Z.$$
Let's observe that, for $x\in K^\times,$ we have:
 $$\deg(x):= \deg(xA)=-d_\infty v_\infty(x).$$

\medskip

Let $I\in \mathcal I(A),$ then there exists an integer $h \geq 1$ such that $I^h=xA, $ $x\in K^\times.$ We set:
$$\langle I \rangle:= \langle x \rangle^{\frac{1}{h}}\in U_\infty.$$
Then one shows (see \cite{GOS}, Section 8.2) that the map $[\cdot]: \mathcal I(A) \rightarrow \overline{K}_\infty^\times, I\mapsto \langle I \rangle\pi^{-\frac{\deg I}{d_\infty}}$ is a group homomorphism such that:
 $$\forall x\in K^\times, \quad [xA]=\frac{x}{\sgn(x)}.$$
Observe that:
$$\forall I\in \mathcal I(A), \quad \sgn ([I])=1.$$
If $M$ is a finite $A$-module, and ${\rm Fitt}_A (M)$ is the Fitting ideal of $M$, we set:
$$[M]_A:=[{\rm Fitt}_A (M)].$$
Let's observe that, if $0\rightarrow M_1\rightarrow M\rightarrow M_2\rightarrow 0$ is a short exact sequence of finite $A$-modules, then:
$$[M]_A=[M_1]_A[M_2]_A.$$

\medskip

Let $E/K$ be a finite extension, and let $O_E$ be the integral closure of $A$ in $E.$ Let $\mathcal I(O_E)$ be the group of non-zero fractional ideals of $O_E.$ We denote by $N_{E/K}:  \mathcal I(O_E)\rightarrow \mathcal I(A)$ the group homomorphism such that, if $\frak P$ is a maximal ideal of $O_E$ and $P=\frak P\cap A,$ we have:
 $$N_{E/K}(\frak P)=P^{[\frac{O_E}{\frak P}: \frac{A}{P}]}.$$
Note that, if $\frak P=xO_E, x\in E^\times,$ then:
    $$N_{E/K} (\frak P) =N_{E/K} (x)A,$$
where $N_{E/K}: E\rightarrow K$ also denotes the usual norm map.

%%%%%%%%%%%%%%%%%%%%%%%%%%%%%%%%%%%%%%%%%

\section{Stark units and $L$-series attached to Drinfeld modules}\label{StarkL}

\subsection{$L$-series attached to Drinfeld modules}\label{L}${}$

\medskip

Let $E/K$ be a finite extension, and let $O_E$ be the integral closure of $A$ in $E.$ Let $\tau:E\rightarrow E, x\mapsto x^q.$ Let $\rho$ be an Drinfeld $A$-module (or a \textit{Drinfeld module} for short) of rank $r\geq 1$ defined over $O_E,$ i.e. $\rho:A \hookrightarrow O_E\{\tau\}$ is an $\mathbb F_q$-algebra homomorphism such that:
 $$\forall a \in A\setminus\{0\}, \quad \rho_a=\rho_{a,0}+\rho_{a,1}\tau +\cdots +\rho_{a,r\deg a}\tau^{r\deg a},$$
where $ \rho_{a,0}, \ldots, \rho_{a,r\deg a}\in O_E,$ $\rho_{a,0}=a,$ and $\rho_{a,r\deg a}\not =0.$

Let $\frak P$ be a maximal ideal of $O_E,$ we denote by $\rho(O_E/\frak P)$ the finite dimensional $\mathbb F_q$-vector space $O_E/\frak P$ equipped with the structure of $A$-module induced by $\rho.$

\begin{proposition}\label{PropositionS3.1}
The following product converges to a principal unit in $K_\infty^\times$ (i.e. an element in $U_\infty\cap K_\infty^\times$):
$$L_A(\rho/O_E):=\prod_{\frak P} \frac{[\frac{O_E}{\frak P}]_A}{[\rho(\frac{O_E}{\frak P})]_A},$$
where $\frak P$ runs through the maximal ideals of $O_E.$
\end{proposition}

\begin{proof}
By \cite{GOS}, Remark 7.1.8.2, we have: $H_A\subset E,$ where $H_A/K$ is the maximal unramified abelian extension of $K$ such that $\infty$ splits completely in $H_A.$ Thus $N_{E/K}(\frak P)$ is a principal ideal. Observe that:
$${\rm Fitt}_A\frac{O_E}{\frak P}= N_{E/K}(\frak P).$$
Thus:
$$[\frac{O_E}{\frak P}]_A=[N_{E/K}(\frak P)].$$
By \cite{GEK}, Theorem 5.1, there exists a unitary polynomial $P(X)\in A[X]$ of degree $r'\leq r$ such that:
\begin{align*}
N_{E/K}(\frak P) &=P(0)A, \\
{\rm Fitt}_A\rho(\frac{O_E}{\frak P})&=P(1)A, \\
v_\infty(\frac{(-1)^{r'}P(0)}{P(1)}-1) &\geq \frac{\deg (N_{E/K}(\frak P))}{r'd_\infty} .
\end{align*}
This last assertion comes from the fact that $P(X)$ is a power of the minimal polynomial over $K$ of the Frobenius $F$ of $\frac{O_E}{\frak P}$  (see \cite{GEK}, Lemma 3.3), and that $K(F)/K$ is totally imaginary (i.e. there exists a unique place of $K(F)$ over $\infty$). By the properties of $[\cdot]$ (see Section \ref{Notation}), we have:
$$\frac{[\frac{O_E}{\frak P}]_A}{[\rho(\frac{O_E}{\frak P})]_A}= \frac{(-1)^{r'}P(0)}{P(1)}.$$
The proposition follows.
\end{proof}

\begin{remark} \label{RemarkS3.1}
The element $L_A(\rho/O_E) \in K_\infty^\times$ is called \textit{the $L$-series} attached to $\rho/O_E.$ By the proof of Proposition \ref{PropositionS3.1}, $L_A(\rho/O_E)$ depends on $A, \rho$ and $O_E,$ but not on the choice of $\pi.$
\end{remark}

Let $F/K$ be a finite extension with $F\subset E,$ and such that  there exists a unique place of $F$ above $\infty$ (still denoted by $\infty$). Let $A'$ be the integral closure of $A$ in $F,$ then $A'$ is the set of elements in $F$ which are regular outside $\infty.$ We assume that $\rho$ extends into a Drinfeld $A'$-module: $\rho: A'\hookrightarrow O_E\{\tau\}.$ Let $[\cdot]_{A'}: \mathcal I(A')\rightarrow \overline{K}_\infty^\times$ be the map constructed as in Section \ref{Notation} with the help of the choice of a uniformizer $\pi'\in F_\infty^\times.$ Let $N_{F_\infty/K_\infty}: F_\infty \rightarrow K_\infty$ be the usual norm map.

\begin{corollary}\label{CorollaryS3.1}
We have:
 $$N_{F_\infty/K_\infty} (L_{A'}(\rho/O_E))=L_A(\rho/O_E).$$
\end{corollary}

\begin{proof}
Recall that:
 $$L_{A'}(\rho/O_E):=\prod_{\frak P} \frac{[\frac{O_E}{\frak P}]_{A'}}{[\rho(\frac{O_E}{\frak P})]_{A'}},$$
where $\frak P$ runs through the maximal ideals of $O_E.$ Since $N_{F_\infty/K_\infty}$ is continuous, we get by the proof of Proposition \ref{PropositionS3.1}:
 $$N_{F_\infty/K_\infty} (L_{A'}(\rho/O_E))=\prod_{\frak P} N_{F_\infty/K_\infty}(\frac{[\frac{O_E}{\frak P}]_{A'}}{[\rho(\frac{O_E}{\frak P})]_{A'}}).$$
Let $\frak P$ be a maximal ideal of $O_E.$ Since $\frac{[\frac{O_E}{\frak P}]_{A'}}{[\rho(\frac{O_E}{\frak P})]_{A'}}\in F^\times,$ we get:
 $$N_{F_\infty/K_\infty}(\frac{[\frac{O_E}{\frak P}]_{A'}}{[\rho(\frac{O_E}{\frak P})]_{A'}})=N_{F/K}(\frac{[\frac{O_E}{\frak P}]_{A'}}{[\rho(\frac{O_E}{\frak P})]_{A'}}).$$
But, observe that if $M$ is a finite $A'$-module, we have:
 $$N_{F/K}({\rm Fitt}_{A'} M)= {\rm Fitt}_{A} M.$$
By the proof of Proposition \ref{PropositionS3.1}, $\frac{[\frac{O_E}{\frak P}]_{A'}}{[\rho(\frac{O_E}{\frak P})]_{A'}}$ is a principal unit in $F_\infty^\times,$ and therefore $N_{F/K}(\frac{[\frac{O_E}{\frak P}]_{A'}}{[\rho(\frac{O_E}{\frak P})]_{A'}})$ is also a principal unit in $K_\infty^\times.$ Again, by the proof of Proposition \ref{PropositionS3.1}, we get:
 $$N_{F/K}(\frac{[\frac{O_E}{\frak P}]_{A'}}{[\rho(\frac{O_E}{\frak P})]_{A'}})= \frac{[\frac{O_E}{\frak P}]_{A}}{[\rho(\frac{O_E}{\frak P})]_{A}}.$$
The corollary follows.
\end{proof}

%%%%%%%%%%%%%%%%%%%%%%%%%%%%%%%%%%%%%%%%%%%%%%%%%%%

\subsection{Stark units and the Taelman class module}\label{S}${}$

\medskip

Let $E/K$ be a finite extension of degree $n$, and let $O_E$ be the integral closure of $A$ in $E.$ Set:
$$E_\infty= E\otimes_KK_\infty.$$
Let $M$ be an $A$-module, $M\subset E_\infty,$ we say that  $M$ is an $A$-lattice in $E_\infty$ if $M$ is discrete and cocompact in $E_\infty.$ Observe that if $M$ is an $A$-lattice in $E_\infty,$ then there exist $e_1, \ldots, e_n \in E_\infty$ (recall that $n=[E:K]$) such that $E_\infty=\oplus_{i=1}^n K_\infty e_i,$ $N:= \oplus _{i=1}^nAe_i \subset M$ and $\frac{M}{N}$ is a finite $A$-module. Note also that $O_E$ is an $A$-lattice in $E_\infty.$\par

Let $\tau: E_\infty\rightarrow E_\infty, x\mapsto x^q.$ Let $\rho: A\hookrightarrow O_E\{\tau\}$ be a Drinfeld module of rank $r\geq 1.$ Then, there exist unique elements $\exp_\rho, \log_\rho \in E\{\{\tau\}\}$ such that
 $$\exp_\rho,\log_\rho \in 1+ E\{\{\tau\}\}\tau,$$
 $$\forall a\in A, \quad \exp_\rho a= \rho_a\exp_\rho,$$
 $$\exp_\rho \log_\rho=\log_\rho \exp_\rho =1.$$
The formal series $\exp_\rho,$ and $\log_\rho$ are respectively called \textit{the exponential series} and \textit{the logarithm series} associated to $\rho/O_E.$ We will write:
 $$\exp_\rho=\sum_{i\geq 0} e_i(\rho) \tau ^i,$$
 $$\log_\rho=\sum_{i\geq 0} l_i(\rho) \tau ^i,$$
with $e_i(\rho), l_i(\rho)\in E.$ Moreover, $\exp_\rho$ converges on $E_\infty$ (see \cite{GOS}, proof of Theorem 4.6.9).

\begin{definition} \label{unit module}
We define \textit{the Taelman module of units} associated to $\rho/O_E$ as follows:
 $$U(\rho/O_E)=\{ x\in E_\infty, \exp_\rho (x)\in O_E\}.$$
\end{definition}

Then, as a consequence of \cite{TAE1}, Theorem 1, the $A$-module $U(\rho/O_E)$ is an $A$-lattice in $E_\infty.$

\begin{definition} \label{class module}
We define \textit{the Taelman class module} associated to $\rho/O_E$ by:
 $$H(\rho/O_E)= \frac{E_\infty}{O_E+\exp_\rho (E_\infty)}.$$
\end{definition}

Note that $H(\rho/O_E)$ is an $A$-module via $\rho,$ and by \cite{TAE1}, Theorem 1, $H(\rho/O_E)$ is a finite $A$-module.\par
Let $z$ be an indeterminate over $K_\infty,$ and let $\mathbb T_z(K_\infty)$ be the Tate algebra in the variable $z$ with coefficients in $K_\infty.$ We set:
 $$\mathbb T_z(E_\infty)=E\otimes_K\mathbb T_z(K_\infty).$$
Observe that $E_\infty\subset \mathbb  T_z(E_\infty),$ and $\mathbb T_z(E_\infty)$ is a free $\mathbb T_z(K_\infty)$-module of rank $[E:K].$ Let $\tau: \mathbb  T_z(E_\infty)\rightarrow \mathbb T_z(E_\infty)$ be the continuous $\mathbb F_q[z]$-algebra homomorphism such that:
 $$\forall x\in E_\infty, \quad \tau (x)=x^q.$$
Let $\ev:\mathbb T_z(E_\infty)\twoheadrightarrow E_\infty$ be the surjective $E_\infty$-algebra homomorphism given by:
 $$\forall f\in \mathbb T_z(E_\infty), \quad \ev(f)=f\mid_{z=1}.$$
We have: ${\ker}\,  \ev= (z-1)\mathbb T_z(E_\infty),$ and:
 $$\forall f\in \mathbb T_z(E_\infty), \quad \ev(\tau (f))=\tau (\ev(f)).$$
Recall that:
 $$\exp_\rho=\sum_{i\geq 0} e_i(\rho) \tau ^i, \quad \text{with } e_i(\rho)\in E.$$
We set:
 $$\exp_{\widetilde{\rho}}= \sum_{i\geq 0} e_i(\rho) z^i \tau ^i\in E[z]\{\{\tau\}\}.$$
Observe that $\exp_{\widetilde{\rho}}$ converges on $\mathbb T_z(E_\infty),$ and:
 $$\forall f\in \mathbb T_z(E_\infty), \quad \ev(\exp_{\widetilde{\rho}}(f))= \exp_\rho (\ev(f)).$$
Let $\widetilde{\rho}:A\hookrightarrow O_E[z]\{\tau\}$ be the $\mathbb F_q$-algebra homomorphism given by:
$$\forall a \in A, \quad \widetilde{\rho}_a=a+\rho_{a,1} z\tau +\cdots +\rho_{a,r\deg a}z^{r\deg a}\tau^{r\deg a},$$
where $\rho_a=a+\rho_{a,1}\tau +\cdots +\rho_{a,r\deg a}\tau^{r\deg a}.$ Then:
 $$\forall a\in A, \quad \exp_{\widetilde{\rho}} a= \widetilde{\rho}_a  \exp_{\widetilde{\rho}}.$$

\begin{definition} \label{Stark units}
\textit{The module of $z$-units} associated to $\rho/O_E$ is defined by:
 $$U(\widetilde{\rho}/O_E[z])= \{ f\in \mathbb T_z(E_\infty), \exp_{\widetilde{\rho}}(f)\in O_E[z]\}.$$
And \textit{the module of Stark units} associated to $\rho/O_E$ is defined by:
 $$U_{St}(\rho/O_E) := \ev(U(\widetilde{\rho}/O_E[z])).$$
\end{definition}

Observe that $U_{St}(\rho/O_E)\subset U(\rho/O_E).$

\begin{theorem}\label{TheoremS4.1}
The $A$-module $U_{St}(\rho/O_E)$ is an $A$-lattice in $E_\infty.$ Furthermore:
$$[\frac{U(\rho/O_E)}{U_{St}(\rho/O_E)}]_A=[H(\rho/O_E)]_A.$$
\end{theorem}

\begin{proof}
This is a consequence of the proof of \cite{ATR}, Theorem 1.  For the convenience of the reader, we give a sketch of the proof. Let's set:
 $$H(\widetilde{\rho}/O_E[z])= \frac{\mathbb T_z (E_\infty)}{O_E[z]+\exp_{\widetilde{\rho}}(\mathbb T_z(E_\infty))}.$$
Observe that $H(\widetilde{\rho}/O_E[z])$ is an $A[z]$-module via $\widetilde{\rho},$ and furthermore $H(\widetilde{\rho}/O_E[z])$ is a finite $\mathbb F_q[z]$-module (\cite{ATR}, Proposition 2). Let's set:
 $$V=\{ x\in H(\widetilde{\rho}/O_E[z]), (z-1)x=0\}.$$
Since ${\ker}\,  \ev = (z-1)\mathbb T_z(E_\infty),$ the multiplication by $z-1$ on $H(\widetilde{\rho}/O_E)$ gives rise to an exact sequence of finite $A$-modules:
 $$0\rightarrow V\rightarrow H(\widetilde{\rho}/O_E[z])\rightarrow H(\widetilde{\rho}/O_E[z])\rightarrow H(\rho/O_E)\rightarrow 0.$$
Thus:
 $${\rm Fitt}_A V= {\rm Fitt}_A H(\rho/O_E).$$
Now, let's consider the homomorphism of $\mathbb F_q[z]$-modules $\alpha: \mathbb T_z(E_\infty) \rightarrow \mathbb T_z(E_\infty)$ given by:
 $$\forall x\in \mathbb T_z(E_\infty), \quad \alpha (x) =\frac{\exp_{\widetilde{\rho}}(x)-\exp_\rho(x)}{z-1}.$$
Observe that:
 $$(z-1)\alpha (U(\rho/O_E))\subset O_E+\exp_{\widetilde{\rho}}(\mathbb T_z(E_\infty)),$$
 $$\forall a\in A, \forall x\in U(\rho (O_E)), \quad \alpha (ax)-\widetilde{\rho}_a(\alpha(x)) \in O_E[z].$$
Thus $\alpha$ induces a homomorphism of $A$-modules:
 $$\bar {\alpha} : U(\rho/O_E)\rightarrow V.$$
By \cite{ATR}, Proposition 3, this homomorphism is surjective and its kernel is precisely $U_{St}(\rho/O_E).$ The theorem follows.
\end{proof}

%%%%%%%%%%%%%%%%%%%%%%%%%%%%%%%%%%%%%%%%%%%%%%%%%%%%

\subsection{Co-volumes}\label{Prel}${}$

%\label{R}${}$

\medskip

%\subsection{Preliminaries}\label{Prel}${}$\par

Let $V$ be a finite dimensional $K_\infty$-vector space of dimension $n \geq 1.$ An $A$-lattice in $V$ is a discrete and cocompact sub-$A$-module of $V.$

\begin{lemma}\label{LemmaS5.1}
Let $M, N$ be two $A$-lattices in $V.$ Then there exists an isomorphism of $K_\infty$-vector spaces $\sigma: V\rightarrow V$ such that:
$$\sigma (M)\subset N.$$
\end{lemma}

\begin{proof}
Since $A$ is a Dedekind domain, there exist two non-zero ideals $I,J$ of $A,$ and two $K_\infty$-basis $\{e_1, \ldots e_n\}, \{f_1, \ldots, f_n\}$ of $V,$ such that:
 $$M=\oplus_{j=1}^{n-1} Ae_j\oplus Ie_n,$$
 $$N=\oplus_{j=1}^{n-1} Af_j\oplus Jf_n.$$
Furthermore, $M$ and $N$ are isomorphic as $A$-modules if and only if $I$ and $J$ have the same class in the ideal class group ${\rm Pic}(A)$ of $A$. Let $x\in I^{-1}J\setminus\{0\}.$ Let $\sigma: V\rightarrow V$ such that:
\begin{align*}
\sigma(e_j) &=f_j, \quad j=1, \ldots, n-1, \\
\sigma (e_n) &= xf_n.
\end{align*}
Then:
 $$\sigma (M)\subset N.$$
Note that if $M$ and $N$ are isomorphic $A$-modules then we can select $x\in K^\times$ such that $I^{-1}J= xA$ and in this case $\sigma (M)=N.$
\end{proof}

\begin{lemma}\label{LemmaS5.2}
Let $M, N$ be two $A$-lattices in $V.$ Let  $\sigma_1, \sigma_2: V\rightarrow V$ be two isomorphisms of $K_\infty$-vector spaces such that $\sigma_i (M)\subset N, i=1,2.$ Then:
$$\frac{\det_{K_\infty} \sigma_1}{\sgn (\det_{K_\infty} \sigma_1)} [\frac{N}{\sigma_1(M)}]_A^{-1}=\frac{\det_{K_\infty} \sigma_2}{\sgn (\det_{K_\infty} \sigma_2)} [\frac{N}{\sigma_2(M)}]_A^{-1}.$$
\end{lemma}

\begin{proof}
Let $\sigma=\sigma_1\sigma_2^{-1}.$ Since $\sigma(\sigma_2(M))=\sigma_1(M)\subset N$, with $\sigma_2(M)\subset N$, we can find $a\in A$ with $\sgn a =1$ such that $a\sigma(N)\subset N$. Set $U = \frac 1a \sigma_2(M) \cap N$.
Then the multiplication by $a$ induces an exact sequence of finite $A$-modules:
$$ 0\longrightarrow \frac U{\sigma_2(M)} \longrightarrow \frac{N}{\sigma_2(M)} \overset a \longrightarrow \frac{N}{\sigma_2(M)} \longrightarrow \frac{N}{aN} \longrightarrow 0$$
from which we deduce
$$[\frac{U}{\sigma_2(M)}]_A = [\frac{N}{aN}]_A = a^n.$$
And $a\sigma$ similarly induces an exact sequence of finite $A$-modules:
$$ 0\longrightarrow \frac U{\sigma_2(M)} \longrightarrow \frac{N}{\sigma_2(M)} \overset {a\sigma}\longrightarrow \frac{N}{\sigma_1(M)} \longrightarrow \frac{N}{a\sigma(N)} \longrightarrow 0.$$
We get:
\begin{eqnarray*}
[\frac{N}{\sigma_1(M)}]_A&=&[\frac{N}{\sigma_2(M)}]_A[\frac{U}{\sigma_2(M)}]_A^{-1} [\frac{N}{a\sigma(N)}]_A = [\frac{N}{\sigma_2(M)}]_A a^{-n} \frac{\det_{K_\infty}(a\sigma)}{\sgn(\det_{K_\infty}(a\sigma))}\\
& =&  [\frac{N}{\sigma_2(M)}]_A\frac{\det_{K_\infty}(\sigma)}{\sgn(\det_{K_\infty}(\sigma))}.
\end{eqnarray*}
The lemma follows.
\end{proof}

Let $M, N$ be two $A$-lattices in $V.$ By Lemma \ref{LemmaS5.1}, there exists an isomorphism of $K_\infty$-vector spaces $\sigma:V\rightarrow V$ such that $\sigma(M)\subset N,$ we set:
 $$[M:N]_A=\frac{\det_{K_\infty} \sigma}{\sgn (\det_{K_\infty} \sigma)} [\frac{N}{\sigma(M)}]_A^{-1}.$$
By Lemma \ref{LemmaS5.2}, this is well-defined.
In particular, if $M,N$ are two $A$-lattices in $V$ such that $N\subset M,$ then:
 $$[M:N]_A=[\frac{M}{N}]_A.$$
If $M,N,U$ are three $A$-lattices in $V,$ we get:
 $$[M:N]_A=[M:U]_A [U:N]_A.$$
\par
Let $F/K$ be a finite extension  such that  there exists a unique place of $F$ above $\infty$ (still denoted by $\infty$). Let $A'$ be the integral closure of $A$ in $F.$ We assume that $V$ is also an $F_\infty$-vector space. Let $[\cdot]_{A'}: \mathcal I(A')\rightarrow \overline{K}_\infty^\times$ be the map constructed as in Section \ref{Notation} with the help of the choice of a uniformizer $\pi'\in F_\infty^\times.$ Let $N_{F_\infty/K_\infty}: F_\infty \rightarrow K_\infty$ be the usual norm map.

\begin{lemma}\label{LemmaS5.3}
Let $M, N$ be two $A'$-lattices in $V.$ Then there exists an integer $m\geq 1$ such that $[M:N]_{A'}^{m} \in F_\infty^\times,[M:N]_A^{m}\in K_\infty^\times,$ and:
 $$N_{F_\infty/K_\infty}([M:N]_{A'}^{m})= [M:N]_A^{m}.$$
\end{lemma}

\begin{proof}
Let $\sigma:V\rightarrow V$ be an isomorphism of $F_\infty$-vector spaces such that $\sigma(M)\subset N,$ and we set: $I'={\rm Fitt}_{A'}\frac{N}{\sigma(M)}.$ Then:
 $${\rm Fitt}_{A}\frac{N}{\sigma(M)}=N_{F/K}(I').$$
Let $m \geq 1$ be an integer such that:
 $$I'^m=xA', \quad x\in A'\setminus\{0\}.$$
Then:
 $$[M:N]_{A'}^m= (\frac{\det_{F_\infty} \sigma}{\sgn' (\det_{F_\infty} \sigma)})^m\frac{\sgn'(x)}{x}.$$
Furthermore, we have:
 $${\det}_{K_\infty} \sigma= N_{F_\infty/K_\infty}({\det}_{F_\infty} \sigma).$$
Thus:
 $$[M:N]_A^m= (\frac{N_{F_\infty/K_\infty}({\det}_{F_\infty} \sigma)}{\sgn(N_{F_\infty/K_\infty}(\rm{det}_{F_\infty} \sigma))})^m \frac{\sgn(N_{F/K}(x))}{N_{F/K}(x)}.$$
Therefore:
 $$N_{F_\infty/K_\infty}([M:N]_{A'}^m)\in [M:N]_A^m \mathbb F_\infty^\times.$$
The lemma follows.
\end{proof}

%%%%%%%%%%%%%%%%%%%%%%%%%%%%%%%%%%%%%

\subsection{Regulator of Stark units and $L$-series}\label{RegS}${}$

\medskip

Let $E/K$ be a finite extension, $E\subset \mathbb C_\infty.$ Recall that $E_\infty=E\otimes_KK_\infty.$ If $M$ is an $A$-lattice in $E_\infty,$ then we call $[O_E:M]_A$ the \emph{$A$-regulator of $M$}.

\begin{definition}
Let $\rho:A\hookrightarrow O_E\{\tau\}$ be a Drinfeld module of rank $r\geq1.$ We define \textit{the regulator of Stark units} associated to $\rho/O_E$ by $[O_E:U_{St}(\rho/O_E)]_A.$
\end{definition}

\begin{proposition}\label{PropositionS5.1}
Let $\rho:A\hookrightarrow O_E\{\tau\}$ be a Drinfeld module of rank $r\geq1.$ We have:
 $$[O_E:U_{St}(\rho/O_E)]_A\in U_\infty.$$
Furthermore, the regulator of Stark units relative to $\rho/O_E$  depends on $\rho, A$ and $O_E,$ not on the choice of $\pi.$
\end{proposition}

\begin{proof}
Let $\theta \in A\setminus\mathbb F_q,$ and let $L=\mathbb F_q(\theta), B=\mathbb F_q[\theta].$ Let $[\cdot]_B: \mathcal I(B)\rightarrow L_\infty^\times$ be the map as in Section \ref{Notation} associated to the choice of $\frac{1}{\theta}$ as a uniformizer of $L_\infty.$ Then, by  Theorem \ref{TheoremS4.1}, we have:
 $$[O_E:U_{St}(\rho/O_E)]_B= [O_E:U(\rho/O_E)]_B [H(\rho/O_E)]_B.$$
Then, by \cite{TAE1}, Theorem 2, we get:
 $$[O_E:U_{St}(\rho/O_E)]_B\in 1+\frac{1}{\theta} \mathbb F_q[[\frac{1}{\theta}]].$$
Now, by Lemma \ref{LemmaS5.3}, there exists an integer $m\geq 1$ such that:
 $$N_{K_\infty/L_\infty}([O_E:U_{St}(\rho/O_E)]_A^m)=[O_E:U_{St}(\rho/O_E)]_B^m.$$
This implies:
 $$v_\infty([O_E:U_{St}(\rho/O_E)]_A)=0.$$
Thus:
 $$[O_E:U_{St}(\rho/O_E)]_A\in \overline{\mathbb F}_q^\times\times U_\infty.$$
But $\sgn ([O_E:U_{St}(\rho/O_E)]_A)=1,$ thus:
 $$[O_E:U_{St}(\rho/O_E)]_A\in  U_\infty.$$\par
Let $\pi'$ be another uniformizer of $K_\infty,$ and let $[\cdot]'_A:\mathcal I(A)\rightarrow \overline{K}_\infty^\times$ be the map as in Section \ref{Notation} associated to $\pi'.$ Then, by the above discussion, we get:
 $$[O_E:U_{St}(\rho/O_E)]'_A\in  U_\infty.$$
Again, by Lemma \ref{LemmaS5.3}, there exists an integer $m'\geq 1$ such that:
 $$([O_E:U_{St}(\rho/O_E)]'_A)^{m'}=[O_E:U_{St}(\rho/O_E)]_A^{m'}.$$
Since $[O_E:U_{St}(\rho/O_E)]'_A,[O_E:U_{St}(\rho/O_E)]_A\in  U_\infty,$ we get:
 $$[O_E:U_{St}(\rho/O_E)]'_A=[O_E:U_{St}(\rho/O_E)]_A.$$
This concludes the proof of the proposition.
\end{proof}

Let's set:
 $$\alpha_A(\rho/O_E):=\frac{L_A(\rho/O_E)}{[O_E:U_{St}(\rho/O_E)]_A}\in \overline{K}_\infty^\times.$$
By Proposition \ref{PropositionS5.1} and Remark \ref{RemarkS3.1}, $\alpha_A(\rho/O_E)$ depends on $A, \rho,$ and $O_E,$ not on the choice of $\pi.$ Furthermore:
 $$\alpha_A(\rho/O_E)\in U_\infty.$$
Let's also observe that, if $p^k$ is the exact power of $p$ dividing $\mid {\rm Pic}(A)\mid,$ then:
 $$\alpha_A(\rho/O_E)^{p^kd_\infty}\in K_\infty^\times.$$
We have the fundamental result due to L. Taelman (\cite{TAE2}, Theorem 1):

\begin{theorem}[Taelman] \label{TheoremS5.1}
Assume that the genus of $K$ is zero and $d_\infty=1.$ Then:
$$\alpha_A(\rho/O_E)=1.$$
\end{theorem}

\begin{proof}
Select $\theta\in A\setminus\mathbb F_q$ such that $v_\infty(\theta)=1.$ Then $A=\mathbb F_q[\theta].$ Let $[\cdot]_A:\mathcal I(A)\rightarrow K_\infty^\times$ be the map as in Section \ref{Notation} associated to the choice of $\frac{1}{\theta}$ as a uniformizer of $K_\infty.$ Then, by Proposition \ref{PropositionS5.1}, Theorem \ref{TheoremS4.1} and \cite{TAE2}, Theorem 1:
$$[O_E:U_{St}(\rho/O_E)]_A= [O_E:U(\rho/O_E)]_A[H(\rho/O_E)]_A= L_A(\rho/O_E).$$
This concludes the proof of the theorem.
\end{proof}

\begin{corollary}\label{CorollaryS5.1} ${}$\par
\noindent 1) Let $F/K$ be a finite extension, $F\subset E,$ and such that  there exists a unique place of $F$ above $\infty$ (still denoted by $\infty$). Let $A'$ be the integral closure of $A$ in $F.$ Let $N_{F_\infty/K_\infty}: F_\infty \rightarrow K_\infty$ be the usual norm map. Then, there exists an integer $k\geq 1$ such that $\alpha_{A'}(\rho/O_E)^k\in F_\infty^\times, \alpha_A(\rho/O_E)^k\in K_\infty^\times,$ and:
$$N_{F_\infty/K_\infty}(\alpha_{A'}(\rho/O_E)^k)=\alpha_A(\rho/O_E)^k.$$
In particular, $\alpha_{A'}(\rho/O_E)=1\Rightarrow \alpha_A(\rho/O_E)=1.$

\medskip

\noindent 2) If there exists an integer $m\geq 1$ such that $\alpha_A(\rho/O_E)^m \in K^\times,$ then $\alpha_A(\rho/O_E)=1.$ In particular, if $\sigma (\alpha_A(\rho/O_E))= \alpha_A(\rho^{\sigma}/\sigma (O_E))$ for all $\sigma \in {\rm Aut}_K(\mathbb C_\infty),$ then $\alpha_A(\rho/O_E)=1.$
\end{corollary}

\begin{proof}${}$\par
\noindent 1) The first assertion  is a consequence of Corollary \ref{CorollaryS3.1} and Lemma \ref{LemmaS5.3}. If $\alpha_{A'}(\rho/O_E)=1,$ then there exists an integer $k\geq 1$ such that $\alpha_A(\rho/O_E)^k=1.$ But, since $\sgn(\alpha_A(\rho/O_E))=1,$ we get $\alpha_A(\rho/O_E)=1.$

\medskip

\noindent 2) Let $x= \alpha_A(\rho/O_E)^m \in K^\times.$ Let $P$ be a maximal ideal of $A,$ and select an integer $l\geq 1$ such that $P^l$ is a principal ideal. Let $\theta \in A\setminus \mathbb F_q$ such that $P^l=\theta A.$ Let $L=\mathbb F_q(\theta)$ and $B=\mathbb F_q[\theta].$ Then, by Taelman's Theorem (Theorem \ref{TheoremS5.1}), we have:
 $$\alpha_B(\rho/O_E)=1.$$
Therefore, by 1), we have:
 $$N_{K/L}(x)\in \mathbb F_q^\times.$$
Since $P$ is the only maximal ideal of $A$ above $\theta B,$ we deduce that $x$ is a $P$-adic unit. Since this is true for all maximal ideal of $A,$ we get:
 $$x\in \mathbb  F_q^\times.$$
But, $\sgn(\alpha_A(\rho/O_E))=1,$ thus: $\alpha_A(\rho/O_E)=1.$

\medskip

\noindent Let's assume that $\sigma (\alpha_A(\rho/O_E))= \alpha_A(\rho^{\sigma}/\sigma (O_E))$ for all $\sigma \in {\rm Aut}_K(\mathbb C_\infty).$ Let $\sigma \in
{\rm Aut}_K(\mathbb C_\infty).$  Let $\frak P$ be a maximal ideal of $O_E,$ then:
 $$[\frac{\sigma(O_E)}{\sigma (\frak P)}]_A=[\frac{O_E}{\frak P}]_A,$$
 $$[\rho^{\sigma}(\frac{\sigma(O_E)}{\sigma (\frak P)})]_A=[\rho (\frac{O_E}{\frak P})]_A.$$
Thus:
 $$L_A(\rho^{\sigma}/\sigma(O_E))=L_A(\rho/O_E).$$
Observe that $\sigma$ induces a $K_\infty$-algebra isomorphism:
 $$E_\infty\simeq \sigma(E)_\infty.$$
Note that $\exp_{\widetilde{\rho}}: E[[z]]\rightarrow E[[z]]$ is an $\mathbb F_q[[z]]$-algebra isomorphism. Therefore:
 $$U(\widetilde{\rho}/O_E[z])\subset E[[z]].$$
Thus:
 $$U(\widetilde{\rho^{\sigma}}/\sigma(O_E)[z])=\sigma (U(\widetilde{\rho}/O_E[z])).$$
By the definition of Stark units, we get:
 $$U_{St}(\rho^{\sigma}/\sigma(O_E))=\sigma(U_{St}(\rho/O_E)).$$
Thus:
 $$[\sigma(O_E): U_{St}(\rho^{\sigma}/\sigma(O_E))]_A=[O_E:U_{St}(\rho/O_E)]_A.$$
Therefore:
 $$\alpha_A(\rho^{\sigma}/\sigma (O_E))=\alpha_A(\rho/O_E).$$
We get:
$$\forall \sigma \in {\rm Aut}_K(\mathbb C_\infty), \quad \sigma (\alpha_A(\rho/O_E))= \alpha_A(\rho/O_E).$$
This implies that $\alpha_A(\rho/O_E)$ is algebraic over $K$ and that there exists an integer $k\geq 0$ such that:
$$\alpha_A(\rho/O_E)^{p^k}\in K^\times.$$
Therefore:
$$\alpha_A(\rho/O_E)=1.$$
\end{proof}

We do not know whether $\alpha_A(\rho/O_E)$ is algebraic over $K,$ and  it might be too naive to expect that $\alpha_A(\rho/O_E)=1$ in general. However, in the next section, we will prove that, if $\phi$ is a sign-normalized rank one Drinfeld module and $E/K$ is a finite abelian extension such that $H \subset E,$ then $\alpha_A(\phi/O_E)=1$ (Theorem \ref{TheoremS6.3}). L. Taelman informed us that the class formula (\cite{TAE2}, Theorem 1) has been generalized by C. Debry to the case where $A$ is a principal ideal domain.

\medskip

We also prove below that $\alpha_A(\phi/O_E)$ is invariant under isogeny, which could be considered as an analogue of the isogeny invariance of the Birch and Swinnerton-Dyer conjecture due to Tate \cite{TAT1}:

\begin{theorem}\label{TheoremS5.2}
Let $E/K$ be a finite extension and let $\rho, \phi:A\rightarrow O_E\{\tau\}$ be two Drinfeld $A$-modules such that there exists $u\in O_E\{\tau\}\setminus \{0\}$ with the following property:
$$\forall a\in A, \quad \rho_au=u\phi_a,$$
then:
 $$\alpha_A(\rho/O_E)=\alpha_A(\phi/O_E).$$
\end{theorem}

\begin{proof} Let $\frak P$ be a maximal ideal of $O_E$ such that $u\not \equiv 0\pmod{\frak P}.$ Then by \cite{GEK}, Theorem 3.5 and Theorem 5.1, we get:
 $$[\rho(\frac{O_E}{\frak P})]_A=[\phi(\frac{O_E}{\frak P})]_A.$$
This implies that there exists an ideal $I\in \mathcal I(A)$ such that:
 $$\frac{L_A(\rho/O_E)}{L_A(\phi/O_E)}=[I].$$
Let $\zeta\in O_E\setminus\{0\}$ be the constant coefficient of $u.$ Then we have the following equality in $E\{\{\tau \}\}:$
 $$\exp_\rho \zeta =u\exp_\phi.$$
Thus:
 $$\exp_{\widetilde{\rho}} \zeta =\widetilde{u}\exp_{\widetilde{\phi}},$$
where, if $u=\sum_{i=0}^m u_i \tau^i, u_i \in O_E,$ $\widetilde{u}=\sum_{i=0}^m u_i z^i \tau^i.$ This implies that:
 $$\zeta U(\widetilde{\phi}/O_E[z])\subset U(\widetilde{\rho}/O_E[z]).$$
Therefore:
 $$\zeta U_{St}(\phi/O_E)\subset U_{St}(\rho/O_E).$$
We get:
 $$[O_E:\zeta U_{St}(\phi/O_E)]_A= [O_E:U_{St}(\phi/O_E)]_A[\frac{O_E}{\zeta O_E}]_A,$$
and:
 $$[O_E: \zeta U_{St}(\phi/O_E)]_A= [O_E:U_{St}(\rho/O_E)]_A[\frac{U_{St}(\rho/O_E)}{\zeta U_{St}(\phi/O_E)}]_A.$$
Therefore, there exists an element $J\in \mathcal I(A)$ such that:
 $$\frac{[O_E:U_{St}(\rho/O_E)]_A}{[O_E:U_{St}(\phi/O_E)]_A}=[J].$$
Finally, we get:
$$\frac{\alpha_A(\rho/O_E)}{\alpha_A(\phi/O_E)}= [IJ^{-1}].$$
Let $x=(\frac{\alpha_A(\rho/O_E)}{\alpha_A(\phi/O_E)})^{h(q^{d_\infty}-1)}\in K^\times,$ where $h=\mid {\rm Pic}(A)\mid.$ Then, by Corollary \ref{CorollaryS5.1}, and Theorem \ref{TheoremS5.1},  if $\theta \in A\setminus \mathbb F_q,$ there exists an integer $k\geq 1$ such that:
$$N_{K_\infty/\mathbb F_q((\frac{1}{\theta}))}(x^k)=1.$$
But, by Proposition \ref{PropositionS5.1}, $x$ is a principal unit in $K_\infty,$ thus:
$$N_{K/\mathbb F_q(\theta)}(x)=1.$$
The above equality being valid for any $\theta\in A\setminus\mathbb F_q$, by the proof of Corollary \ref{CorollaryS5.1}, we deduce that:
$$x=1.$$
Since $\sgn(\frac{\alpha_A(\rho/O_E)}{\alpha_A(\phi/O_E)})=1,$ we get:
$$\frac{\alpha_A(\rho/O_E)}{\alpha_A(\phi/O_E)}=1.$$
\end{proof}
%%%%%%%%%%%%%%%%%%%%%%%%%%%%%%%%%

%\section{Equivariant $A$-harmonic series}\label{E}

\section{Stark units associated to sign-normalized rank one Drinfeld modules}

%%%%%%%%%%%%%%%%%%%%%%%%%%%%%%%%%

\subsection{Zeta functions}\label{Z}${}$

\medskip

In this section, we briefly recall the definition of some zeta functions (\cite{GOS}, Chapter 8).

Recall that if $I\in \mathcal I(A)$, we have set:
 $$[I]=\langle I \rangle\pi^{-\frac{\deg I}{d_\infty}}\in \overline{K}_\infty^\times,$$
where $v_\infty(\langle I \rangle-1)>0,$ and:
 $$\forall x\in K^\times, \quad \langle xA \rangle=\frac{x}{\sgn(x)}\pi^{-v_\infty(x)}.$$
Let $\mathbb S_\infty=\mathbb C_\infty^\times \times \mathbb Z_p$ be the Goss ``complex plane". The group action of $\mathbb S_\infty$ is written additively. Let $I\in \mathcal I(A)$ and $s=(x;y)\in \mathbb S_\infty,$ we set:
 $$I^s= \langle I \rangle^yx^{\deg I}\in \mathbb C_\infty^\times.$$
We have a natural injective group homomorphism: $\mathbb Z\rightarrow \mathbb S_\infty, j\mapsto s_j=(\pi^{-\frac{j}{d_\infty}},j).$ Observe that:
 $$\forall j\in \mathbb Z, \forall I\in \mathcal I(A), \quad I^{s_j}=[I]^j.$$
Let $E/K$ be a finite extension, and let $O_E$ be the integral closure of $A$ in $E.$ Let $\frak I$ be a non-zero ideal of $E$. We have:
 $$\forall j\in \mathbb Z, \quad N_{E/K}(\frak I)^{s_j}=[\frac{O_E}{\frak I}]_A^j.$$
Let $s\in \mathbb S_\infty,$ then the following sum converges in $\mathbb C_\infty$ (\cite{GOS}, Theorem 8.9.2):
 $$\zeta_{O_E}(s):=\sum_{d \geq 0} \sum_{\substack{\frak I\in \mathcal I(O_E), \frak I\subset O_E, \\ \deg(N_{E/K}(\frak I))=d}} N_{E/K}(\frak I)^{-s}.$$
The function $\zeta_{O_E}:\mathbb S_\infty\rightarrow \mathbb C_\infty$ is called the \emph{zeta function attached to $O_E$ and $[\cdot]_A.$} Observe that:
 $$\forall j\in \mathbb Z, \quad \zeta_{O_E}(j):= \zeta_{O_E}(s_j)=\sum_{d\geq 0} \sum_{\substack{\frak I\in \mathcal I(A), \frak I\subset O_E,\\ \deg(N_{E/K}(\frak I))=d}} [\frac{O_E}{\frak I}]_A^{-j}.$$
In particular:
$$\zeta_{O_E}(1)=\prod_{\frak P}(1-\frac{1}{[\frac{O_E}{\frak P}]_A})^{-1}\in \overline{K}_\infty^\times,$$
where $\frak P$ runs through the maximal ideals of $O_E.$

\begin{lemma}\label{LemmaS6.1}
Let $H_A$ be the Hilbert class field of A, i.e. $H_A/K$ is the maximal unramified abelian extension of $A$ in which $\infty$ splits completely. If $H_A\subset E,$ then the function $\zeta_{O_E}(.)$ depends only on $O_E$ and $\sgn\mid_{K_\infty^\times}.$
\end{lemma}
\begin{proof} Let $\frak P$ be a maximal ideal of $O_E.$ Let $A'$ be the integral closure of $A$ in $H_A.$  Let $P'=\frak P\cap A',P=\frak P\cap A.$  By class field theory, $P^{[\frac{A'}{P'}:\frac{A}{P}]}$ is a principal ideal. Thus:
$$N_{E/K}(\frak P)=\theta A,$$
for some $\theta \in A\setminus \mathbb F_q.$ Let $j\in \mathbb  N, j\geq 1.$ We have:
$$(1-\frac{1}{[\frac{O_E}{\frak P}]^j_A})^{-1}= \frac{\frac{\theta^j}{\sgn(\theta^j)}}{\frac{\theta^j}{\sgn (\theta^j)}-1}.$$
But, observe that:
$$\zeta_{O_E}(j)=\prod_{\frak P}(1-\frac{1}{[\frac{O_E}{\frak P}]^j_A})^{-1}\in U_\infty \cap K_\infty^\times.$$
The lemma is thus a consequence of \cite{GOS}, Theorem 8.7.1.
\end{proof}

%%%%%%%%%%%%%%%%%%%%%%%%%

\subsection{Background on sign-normalized rank one Drinfeld modules}\label{sgn}${}$

\medskip

Let $\phi: A\rightarrow \overline{K}_\infty\{\tau\}$ be a rank one Drinfeld module such that there exists $i(\phi)\in \mathbb N$ with the following property: $$\forall a\in A\setminus\{0\}, \quad \phi_a=a+\cdots + \sgn(a)^{q^{i(\phi)}}\tau^{\deg a}.$$
Such a Drinfeld module $\phi$ is said to be \textit{sign-normalized}. By \cite{GOS}, Theorem 7.2.15, there always exist sign-normalized rank one Drinfeld modules.

\medskip

\textbf{From now on, we will fix a sign-normalized rank one Drinfeld module $\phi: A\rightarrow \overline{K}_\infty\{\tau\}$.}

\medskip

Let $I_K$ be the group of id\`eles of $K.$ Let's consider the following subgroup of the id\`eles of $K$ :
 $$K^\times {\ker} \sgn \mid_{K_\infty^\times}\prod_{v\not =\infty} O_v^\times,$$
where for a place $v$ of $K,$ $O_v$ denotes the valuation ring of the $v$-adic completion of $K.$  By class field theory, there exists a unique finite abelian extension $H/K$ such that the reciprocity map  induces an isomorphism:
 $$\frac{I_K}{K^\times {\ker} \sgn \mid_{K_\infty^\times}\prod_{v\not =\infty} O_v^\times}\simeq {\rm Gal} (H/K).$$
The natural surjective homomorphism $I_K\rightarrow \mathcal I(A)$ induces an isomorphism given by the Artin map $(., H/K):$
 $$\frac{\mathcal I(A)}{\mathcal P_+(A)}\simeq {\rm Gal} (H/K),$$
where $\mathcal P_+(A)=\{ xA, x\in K, \sgn(x)=1\}.$ Let $H_A$ be the Hilbert class field of $A,$ i.e. $H_A/K$ corresponds to the following subgroup of the id\`eles of $K:$
$$K^\times K_\infty^\times\prod_{v\not =\infty} O_v^\times.$$
Then $H/K$ is unramified outside $\infty,$ and $H/H_A$ is totally ramified at the places of $H_A$ above $\infty.$ Furthermore:
 $${\rm Gal}(H/H_A)\simeq \frac{\mathbb F_\infty^\times}{\mathbb F_q^\times}.$$
If $w$ is a place of $H$ above $\infty,$ then the $w$-adic completion of $H$ is isomorphic to:
$$K_\infty(((-1)^{d_\infty-1}\pi)^{\frac{q-1}{q^{d_\infty}-1}}).$$
We denote by $B$ the integral closure of $A$ in $H$ and set $A'=B\cap H_A.$ We observe that $\mathbb F_\infty \subset A'.$

We denote by $G$ the Galois group ${\rm Gal}(H/K).$ For $I \in \mathcal I(A),$ we set:
\begin{equation} \label{sigma I}
\sigma_I=(I,H/K) \in G.
\end{equation}

By \cite{GOS}, Proposition 7.4.2 and Corollary 7.4.9, the subfield of $\mathbb C_\infty$ generated by $K$ and the coefficients of $\phi_a$ is $H$. Furthermore (\cite{GOS}, Lemma 7.4.5):
 $$\forall a\in A, \quad \phi_a\in B\{\tau\}.$$
Let $I$ be a non-zero ideal of $A,$ and let's define $\phi_I$ to be the unitary element in  $H\{\tau \}$ such that:
$$H\{ \tau \} \phi_I=\sum_{a\in I} H\{Ê\tau\} \phi_a.$$
We have:
$${\ker} \, \phi_I=\bigcap _{a\in I} {\ker}\, \phi_a,$$
$$\phi_I\in B\{\tau\},$$
$$\deg_\tau \phi_I= \deg I.$$
We write: $\phi_I=\phi_{I,0}+\cdots+\phi_{I,\deg I} \tau^{\deg I}$ with $\phi_{I,\deg I}=1$ and denote by $\psi(I) \in B\setminus\{0\}$ the constant coefficient $\phi_{I,0}$ of $\phi_I.$

\begin{lemma}\label{LemmaS6.2}
The map $\psi$ extends uniquely into a map $\psi: \mathcal I(A)\rightarrow H^\times$ with the following properties:

1) for all $I, J\in \mathcal I(A), \psi(IJ)= \sigma_J(\psi(I))\, \psi(J),$

2) for all $I\in \mathcal I(A), IB=\psi(I)B,$

3) for all $x\in K^\times, \psi (xA)= \frac{x}{\sgn (x)^{q^{i(\phi)}}}.$

\noindent In particular, we have:
 $$\forall x\in K^\times, \quad \sigma_{xA}(\psi(I))= \sgn (x)^{q^{i(\phi)}-q^{i(\phi)+\deg I}} \psi(I).$$
\end{lemma}

\begin{proof}
Let $I\in \mathcal I(A),$ select $a\in A, \sgn (a)=1,$ such that $aI\subset A.$ Let's set:
$$\psi(I):= \frac {\psi (aI)}{a}\in H^\times.$$
By \cite{GOS}, Theorem 7.4.8 and Theorem 7.6.2, the map $\psi: \mathcal I(A)\rightarrow H^\times$ is well-defined and satisfies the desired properties.
\end{proof}

Note that the map $\psi$ determines $H$ and $H_A$:

\begin{proposition}\label{PropositionS6.1}
We have:\par
\noindent 1) $H=K(\psi (I), I\in \mathcal I(A));$\par
\noindent 2) $H_A=K(\psi (I), I\in \mathcal I(A), \deg I\equiv 0\pmod{d_\infty}).$
\end{proposition}

\begin{proof} ${}$\par
\noindent 1) Let $\sigma \in {\rm Gal }(H/K(\psi (I), I\in \mathcal I(A)).$ Let $J\in \mathcal I(A)$ such that $\sigma =\sigma_J.$ Then:
 $$\forall I\in \mathcal I(A), \quad \sigma_I(\psi(J))=\psi(J).$$
Therefore:
 $$\psi(J)\in K^\times.$$
Since $JB=\psi(J) B$ (Lemma \ref{LemmaS6.2}), we get that $J=xA$ for some $x\in K^\times.$ Thus, for all $I\in \mathcal I(A),$ we get:
$$\sgn (x)^{q^{i(\phi)}-q^{i(\phi)+\deg I}}=1.$$
Since $\deg : \mathcal I(A)\rightarrow \mathbb Z$ is a surjective group homomorphism, this implies that $\sgn(x)\in \mathbb F_q^\times$ and thus $J\in \mathcal P_+(A).$ Therefore $\sigma =1.$

\medskip

\noindent 2) Set $E=K(\psi (I), I\in \mathcal I(A), \deg I\equiv 0\pmod{d_\infty}).$ Observe that:
 $${\rm Gal}(H/H_A)=\{ \sigma_{xA}, x\in K^\times\}.$$
Thus:
 $$K(\mathbb F_\infty)\subset E\subset H_A.$$
We also have:
 $${\rm Gal}(H_A/K(\mathbb F_\infty))=\{ (I, H_A/K), I\in \mathcal I(A), \deg I\equiv 0\pmod{d_\infty}\}.$$
Let $\sigma \in {\rm Gal}(H_A/E).$ Then, there exists $J\in \mathcal I(A), $ $\deg J\equiv 0\pmod{d_\infty},$ such that $\sigma=(J, H_A/K).$ But for all $I\in \mathcal I(A), \deg I\equiv 0\pmod{d_\infty},$ we have:
$$\psi(IJ)=\sigma(\psi(I))\psi(J)= \psi(I)\psi(J),$$
and therefore:
$$(I, H_A/K)(\psi(J))=\psi(J).$$
This implies:
$$\psi(J)\in K(\mathbb F_\infty)^\times.$$
But:
$$JA[\mathbb F_\infty]=\psi(J)A[\mathbb F_\infty].$$
Thus  $J^{d_\infty}$ is a principal ideal. But:
$$\psi(J^{d_\infty})=\psi(J)^{d_\infty}.$$
In particular:
 $$\psi(J)^{d_\infty \frac{q^{d_\infty}-1}{q-1}}\in K^\times.$$
Thus, if $\delta$ is the Frobenius in ${\rm Gal}(K(\mathbb F_\infty)/K),$ there exists $\zeta \in \mathbb F_\infty^\times$ such that:
 $$\delta(\psi(J))=\zeta \psi(J).$$
Observe that:
 $$N_{\mathbb  F_\infty/\mathbb  F_q}(\zeta)=1.$$
Thus:
 $$\zeta=\frac{\mu}{\delta(\mu)},$$
for some $\mu \in \mathbb F_\infty^\times.$
This implies that:
 $$\psi(J)\mu \in K^\times.$$
Therefore $J$ is a principal ideal and thus  $\sigma =1.$
\end{proof}

We have  the following crucial fact:

\begin{proposition}\label{PropositionS6.2}
Let $E/K$ be a finite extension such that $H\subset E.$ Then:
 $$L_A(\phi/O_E)=\zeta_{O_E}(1).$$
\end{proposition}

\begin{proof} Let $\frak P$ be a maximal ideal of $O_E.$ Let $m=[\frac{O_E}{\frak P}:\frac{A}{P}].$ Then:
 $$N_{E/K}(\frak P)= P^m.$$
Since $H\subset O_E,$ by class field theory, we get:
 $$P^m=\theta A, \quad \text{with} \quad \theta\in A, \sgn(\theta)=1.$$
Since $\phi$ is a rank one Drinfeld module, it implies that:
 $$\phi_\theta\equiv\tau^{m\deg P}\pmod{\frak P}.$$
This implies that:
 $$[\phi(\frac{O_E}{\frak P})]_A=\theta-1 =[\frac{O_E}{\frak P}]_A-1.$$
We get:
 $$L_A(\phi/O_E)=\prod_{\frak P}\frac{[\frac{O_E}{\frak P}]_A}{[\frac{O_E}{\frak P}]_A-1}= \prod_{\frak P}(1-\frac{1}{[\frac{O_E}{\frak P}]_A})^{-1}=\zeta_{O_E}(1).$$
\end{proof}

%%%%%%%%%%%%%%%%%%%%%%%%%%%%%

\subsection{Equivariant $A$-harmonic series: a detailed example}\label{Example}${}$

\medskip

We keep the notation of Section \ref{sgn}. Let $z$ be an indeterminate over $K_\infty,$ and recall that $\mathbb T_z(K_\infty)$ denotes the Tate algebra in the variable $z$ with coefficients in $K_\infty.$ Recall that:
$$H_\infty=H\otimes_KK_\infty,$$
 $$\mathbb T_z(H_\infty)= H\otimes_K\mathbb T_z(K_\infty).$$
For $n\in \mathbb Z,$ we set:
 $$Z_{B}(n;z)=\sum_{d \geq 0} \sum_{\substack{\frak I\in \mathcal I(B), \frak I\subset B,\\ \deg(N_{E/K}(\frak I))=d}} [\frac{O_E}{\frak I}]_A^{-n} z^d.$$
Then, by \cite{GOS}, Theorem 8.9.2,  for all $n\in \mathbb Z, $ $Z_B(n;.)$ defines an entire function on $\mathbb C_\infty,$ and:
 $$\forall n\in \mathbb N, \quad Z_B(-n; z)\in A[z].$$
Observe that:
 $$\forall n\in \mathbb Z, \quad Z_B(n; z)\in \mathbb T_z(K_\infty),$$
and:
 $$\forall n\geq 1, \quad Z_B(n; z)=\prod_{\frak P}(1-\frac {z^{\deg(N_{H/K}(\frak P))}}{[\frac{O_E}{\frak P}]^n_A})^{-1}\in \mathbb T_z(K_\infty)^\times.$$
Finally, we note that:
 $$Z_B(1; 1)=\zeta_B(1).$$

Recall that $G={\rm Gal}(H/K).$ Then $G\simeq {\rm Gal}(H(z)/K(z))$ acts on $\mathbb T_z(H_\infty).$ We denote by $\mathbb T_z(H_\infty)[G]$ the non-commutative group ring where the commutation rule is given by:
$$\forall h,h'\in \mathbb T_z(H_\infty), \forall g,g'\in G, \quad hg.h'g'= hg(h') gg'.$$
Recall that for $I\in \mathcal I(A), $ we have set \eqref{sigma I}:
$$\sigma_I=(I, H/K)\in G.$$

\begin{lemma}\label{LemmaS6.4}
Let $n\in \mathbb Z.$ The following infinite sum  converges in $\mathbb T_z({H}_\infty)[G]:$
 $$\mathcal L(\phi/B; n; z):=\sum_{d \geq 0} \sum_{\substack{I\in \mathcal I(A) , I\subset A,\\ \deg I=d}} \frac{z^{\deg I}}{\psi(I)^n} \sigma_I.$$
Furthermore, for all $n\geq 1,$ we have:
 $$\mathcal L(\phi/B; n; z)=\prod_{P}(1-\frac{z^{\deg P}}{\psi(P)^n}\sigma_P)^{-1}\in (\mathbb T_z(H_\infty)[G])^\times$$
and for all $n\leq 0$:
$$\mathcal L(\phi/B; n; z)\in B[z][G].$$
\end{lemma}

\begin{proof}
Let $n\geq 1.$ First let's observe that for any place $w$ of $H$ above $\infty:$
 $$\lim_{I\subset A, \deg I\rightarrow +\infty}w(\psi(I))=+\infty.$$
Let $P$ be a maximal ideal of $A.$ Note that:
$$\forall k\geq 0, \quad \psi(P^{k+1})= \sigma_P(\psi(P^k))\psi(P)= \sigma_P^k(\psi(P))\psi(P^k).$$
Thus:
$$ \sum_{m\geq 0}\frac{z^{m\deg P}}{\psi(P^m)^n}\sigma_P^m\in \mathbb T_z(H_\infty)[G],$$
and we have:
$$(1-\frac{z^{\deg P}}{\psi(P)^n}\sigma_P)( \sum_{m\geq 0}\frac{z^{m\deg P}}{\psi(P^m)^n}\sigma_P^m)=( \sum_{m\geq 0}\frac{z^{m\deg P}}{\psi(P^m)^n}\sigma_P^m)(1-\frac{z^{\deg P}}{\psi(P)^n}\sigma_P)=1.$$
Thus, we have:
$$(1-\frac{z^{\deg P}}{\psi(P)^n}\sigma_P)^{-1}:= \sum_{m\geq 0}\frac{z^{m\deg P}}{\psi(P^m)^n}\sigma_P^m\in (\mathbb T_z(H_\infty)[G])^\times.$$\par
Let $P,Q$ be two distinct maximal ideals of $A.$ We have:
$$(1-\frac{z^{\deg P}}{\psi(P)^n}\sigma_P)(1-\frac{z^{\deg Q}}{\psi(Q)^n}\sigma_Q)=(1-\frac{z^{\deg Q}}{\psi(Q)^n}\sigma_Q)(1-\frac{z^{\deg P}}{\psi(P)^n}\sigma_P)=(1-\frac{z^{\deg (PQ)}}{\psi(PQ)^n}\sigma_{PQ}).$$
Therefore:
 $$\mathcal L(\phi/B; n; z)=\prod_{P}(1-\frac{z^{\deg P}}{\psi(P)^n}\sigma_P)^{-1}=\sum_{I\in \mathcal I(A), I\subset A}\frac{z^{\deg I}}{\psi(I)^n}\sigma_I\in (\mathbb T_z(H_\infty)[G])^\times.$$\par
Let $n\in \mathbb Z.$ For $d \in \mathbb N$, we set:
 $$S_{\psi,d}(B; n)=\sum_{\substack{I\in \mathcal I(A) , I\subset A,\\ \deg I=d}} \psi(I)^{-n}\sigma_I \in H[G].$$
Let $h$ be the order of $\frac{\mathcal I(A)}{\mathcal P_+(A)}.$ Let $I_1, \ldots, I_h \in \mathcal I(A) \cap A$ be a system of representatives of $\frac{\mathcal I(A)}{\mathcal P_+(A)}.$ Then:
 $$S_{\psi,d}(B; n)=\sum_{j=1}^h \psi(I_j)^{-n}\sigma_{I_j}\sum_{\substack{a\in K^\times, \sgn (a)=1,\\ aI_j\subset A,\\ \deg(aI_j)=d}}a^{-n}.$$
Now, let's assume that $n\leq 0.$ Then, by \cite{ANDTR}, Lemma 3.2, there exists an integer $d_0(n,\psi, H)\in \mathbb N $ such that, for all $d \geq d_0(n,\psi, H),$ for all $j\in \{1, \ldots, h\},$ we have:
$$\sum_{\substack{a\in K^\times, \sgn (a)=1,\\ aI_j\subset A,\\ \deg(aI_j)=d}}a^{-n}=0.$$
Therefore, for $d \geq d_0(n,\psi, H),$ we have:
$$S_{\psi,d}(B; n)=0.$$
Thus:
$$\forall n\in \mathbb N, \quad \mathcal L(\phi/B; -n; z)\in B[z][G].$$
\end{proof}

The element $\mathcal L(\phi/B):=\mathcal L(\phi/B; 1;1)\in (H_\infty[G])^\times$ will be called \textit{the equivariant $A$-harmonic series} attached to $\phi/B.$

\medskip

Note that $\mathcal L(\phi/B; 1; z)$ induces  a $\mathbb T_z(K_\infty)$-linear map $\mathcal L(\phi/B; 1; z): \mathbb T_z(H_\infty)\rightarrow \mathbb T_z(H_\infty).$ Since $\mathbb T_z(H_\infty)$ is a free $\mathbb T_z(K_\infty)$-module of rank $[H:K]$ (recall that $\mathbb T_z(K_\infty)$ is a principal ideal domain), $\det_{\mathbb T_z(K_\infty)}\mathcal L(\phi/B; 1; z)$ is well-defined.
We also observe that $\mathcal L(\phi/B)$ induces a $K_\infty$-linear map $\mathcal L(\phi/B): H_\infty\rightarrow H_\infty,$ and we denote by $\det_{K_\infty} \mathcal L(\phi/B)$ its determinant. Recall that $\ev: \mathbb T_z(H_\infty)\rightarrow H_\infty$ is the $H_\infty$-linear map given by:
 $$\forall f\in \mathbb T_z(H_\infty), \quad \ev(f)=f\mid_{z=1}.$$
Observe that, if $\{e_1, \ldots, e_n\}$ is a $K$-basis of $H/K$ (recall that $n=[H:K]$), then:
 $$H_\infty= \oplus_{i=1}^n K_\infty e_i,$$
 $$\mathbb T_z(H_\infty)= \oplus _{i=1}^n \mathbb T_z(K_\infty) e_i.$$

We deduce that:
$${\det}_{K_\infty} \mathcal L(\phi/B)=\ev({\det}_{\mathbb T_z(K_\infty)}\mathcal L(\phi/B; 1; z)).$$
\begin{theorem}\label{TheoremS6.1} We have:
$${\det}_{\mathbb T_z(K_\infty)}\mathcal L(\phi/B; 1; z)= Z_B(1;z).$$
In particular:
$${\det}_{K_\infty} \mathcal L(\phi/B)=\zeta_B(1).$$
\end{theorem}

\begin{proof}
First, we  recall that, by Lemma \ref{LemmaS6.4},   we have the following equality in $\mathbb T_z(H_\infty)[G]:$
 $$\prod_{P}(1-\frac{z^{\deg P}}{\psi(P)}\sigma_P)^{-1}= \mathcal L(\phi/B; 1; z),$$
where $P$ runs through the maximal ideals of $A,$ and:
 $$(1-\frac{z^{\deg P}}{\psi(P)}\sigma_P)^{-1}= \sum_{n\geq 0} \frac{z^{n\deg P}}{\psi (P^n)}\sigma_{P^n}.$$
By the properties of $\psi$ (Lemma \ref{LemmaS6.2}), we have:
$$\lim_{N \rightarrow +\infty }\prod_{\deg P\geq N}(1-\frac{z^{\deg P}}{\psi(P)}\sigma_P)^{-1}=1.$$
Thus:
$${\det}_{\mathbb T_z(K_\infty)} \mathcal L(\phi/B; 1; z) =\prod_{P}{\det}_{\mathbb T_z(K_\infty)}(1-\frac{z^{\deg P}}{\psi(P)}\sigma_P)^{-1}.$$
Thus, we are led to compute:
$${\det}_{\mathbb  T_z(K_\infty)}(1-\frac{z^{\deg P}}{\psi(P)}\sigma_P).$$
\medskip
But $1-\frac{z^{\deg P}}{\psi(P)}\sigma_P$ induces a $K[z]$-linear map on $H[z].$ Thus:
$${\det}_{\mathbb T_z(K_\infty)}(1-\frac{z^{\deg P}}{\psi(P)}\sigma_P)={\det}_{K[z]}(1-\frac{z^{\deg P}}{\psi(P)}\sigma_P)\mid_{H[z]}.$$
Let $e\geq 1$ be the order of $P$ in $\frac{\mathcal I(A)}{\mathcal P_+(A)}.$ Write $\xi = \frac{z^{\deg P}}{\psi(P)}\sigma_P\mid_{H[z]}$. We have $\xi^e=\frac{z^{e\deg P}}{\psi(P^e)}\in K[z]$. Since $e$ is the order of $\sigma_P$ in $G$, by Dedekind's Theorem $\sigma_P^0$, $\sigma_P$, \dots, $\sigma_P^{e-1}$ are linearly independent over $H(z)$. We deduce that
$X^e-\frac{z^{e\deg P}}{\psi(P^e)}$ is the minimal polynomial of $\xi$ over $K(z)$ and also over $H^{\langle \sigma_P\rangle}(z)$, and that:
$${\det}_{K[z]}(1-\frac{z^{\deg P}}{\psi(P)}\sigma_P)\mid_{H[z]}=(1-\frac{z^{e\deg P}}{\psi(P^e)})^{\frac{[H:K]}{e}}.$$
Now, let $\frak P$ be a maximal ideal of $B$ above $P.$ Then, by class field theory, we have:
$$[\frac{B}{\frak P}: \frac{A}{P}]=e.$$
Therefore:
$$[\frac{B}{\frak P}]_A=\psi(P^e).$$
Thus:
$${\det}_{K[z]}(1-\frac{z^{\deg P}}{\psi(P)}\sigma_P)\mid_{H[z]}= \prod_{\frak P\mid  P}(1-\frac{z^{\deg(N_{H/K}(\frak P))}}{[\frac{B}{\frak P}]_A}).$$ \medskip
Finally, we get:
$${\det}_{\mathbb T_z(K_\infty)}\mathcal L(\phi/B; 1; z)=\prod_{\frak P}(1-\frac{z^{\deg(N_{H/K}(\frak P))}}{[\frac{B}{\frak P}]_A})^{-1},$$
where $\frak P$ runs through the maximal ideals of $B.$
Thus:
$${\det}_{\mathbb T_z(K_\infty)}\mathcal L(\phi/B; 1; z)= Z_B(1;z).$$
Now:
$${\det}_{K_\infty} \mathcal L(\phi/B)=\ev({\det}_{\mathbb T_z(K_\infty)}\mathcal L(\phi/B; 1; z))=\ev(Z_B(1;z))=\zeta_B(1).$$
\end{proof}

Although this is not evident, the above theorem reflects a class formula \`a la Taelman which will be proved in Section \ref{class formula}.

%%%%%%%%%%%%%%%%%%%%%%%%%%%%%

%\subsection{A class formula}\label{CF}${}$\par

\subsection{Stark units}\label{CF}${}$

\medskip

We keep the notation of the previous sections. We will need the following basic result:

\begin{lemma}\label{LemmaS6.5}
Let $L/K$ be a finite extension, and let $O_L$ be the integral closure of $A$ in $L.$ Let $\rho:A\rightarrow O_L\{\tau\}$ be a Drinfeld module of rank $r\geq 1.$ Let $\exp_\rho,\log_\rho \in 1+ L\{\{\tau\}\}\tau$ be such that:
 $$\forall a\in A, \quad \exp_\rho a= \rho_a\exp_\rho,$$
 $$\exp_\rho \log_\rho=\log_\rho \exp_\rho =1.$$
Write:
 $$\exp_\rho=\sum_{i\geq 0} e_i(\rho) \tau ^i,$$
 $$\log_\rho=\sum_{i\geq 0} l_i(\rho) \tau ^i,$$
with $e_i(\rho), l_i(\rho)\in L.$

\medskip

\noindent 1) Let $P$ be a maximal ideal of $A.$ Let $A_P$ be the $P$-adic completion of $A.$ Then:
$$\forall n\geq 0, \quad P^{q^n}e_n(\rho)O_L\subset  PO_L\otimes_AA_P,$$
$$\forall n\geq 0, \quad P^{[\frac{n}{\deg P}]}l_n(\rho)O_L\subset O_L\otimes_AA_P.$$

\medskip

\noindent 2) Let $\sigma: L\hookrightarrow \overline{K}_\infty$ be a field homomorphism such that $\sigma\mid_K={\rm Id}_K.$ Then, there exist $n(\rho, \sigma)\in \mathbb N, C(\rho, \sigma)\in ]0; +\infty[,$  such that:
 $$\forall n\geq n(\rho, \sigma), \quad v_\infty(\sigma(e_n(\rho)))\geq C(\rho, \sigma) nq^n .$$
\end{lemma}

\begin{proof}${}$\par
\noindent 1) Let $\theta\in A\setminus \mathbb F_q$ such that $\theta A_P=PA_P.$ Let $d=r \deg(\theta),$ and let's write:
 $$\rho_\theta= \sum_{j=0}^{d} \rho_{\theta,j} \tau^j.$$
From $ \exp_\rho \theta= \rho_\theta\exp_\rho$, we get:
 $$\forall n\geq 0, \quad (\theta^{q^n}-\theta)e_n(\rho)=\sum_{l=1}^d \rho_{\theta,l}e_{n-l}(\rho)^{q^l}$$
where $e_i=0$ if $i<0$.
Since $e_0(\rho)=1,$ one proves by induction on $n\geq 0$ that
 $$e_n(\rho) \theta^{q^n}\in \theta ^{{\inf}\{ q-1, q^n\}} O_L\otimes_AA_P.$$

\medskip

\noindent Observe that:
$$\forall a\in A, \quad a\log_\rho = \log_\rho \rho_a.$$
Thus:
$$\forall a\in A, \forall n\geq 0, \quad (a-a^{q^n})l_n(\rho)=\sum_{l=1}^{r\deg a} l_{n-l}(\rho) \rho_{a,l}^{q^{n-l}}.$$
Thus, if $n\not \equiv 0\pmod{\deg P},$ we get:
$$l_n(\rho)O_E\otimes_AA_P\subset \sum_{l=1}^{n} l_{n-l}(\rho) O_L\otimes_AA_P.$$
If $n\equiv 0\pmod{\deg P}, $ we have:
$$(\theta-\theta^{q^n})l_n(\rho)=\sum_{l=1}^d l_{n-l}(\rho) \rho_{\theta,l}^{q^{n-l}}.$$
In any case, we get:
$$\theta^{[\frac{n}{\deg P}]}l_n(\rho)\in \sum_{l=1}^n \theta^{[\frac{n-l}{\deg P}]}l_{n-l}(\rho) O_L\otimes_AA_P.$$
Since $l_0(\rho)=1,$ we get the desired second assertion by induction on $n\geq 0.$

\medskip

\noindent 2) This is a consequence of the proof of \cite{GOS}, Theorem 4.6.9. We give a proof for the convenience of the reader. We keep the previous notation. In particular, let $\theta\in A\setminus \mathbb F_q,$ and write:
 $$\rho_\theta=\sum_{j=0}^{r\deg(\theta)} \rho_{\theta,j} \tau ^j, \quad {\rm with} \quad \rho_{\theta,j} \in \overline{K}_\infty.$$
Recall that $\rho_{\theta,0}=\theta.$ Set $d=r\deg(\theta).$ Then:
 $$\forall n\geq 0, \quad (\theta^{q^n}-\theta)e_n(\rho)=\sum_{l=1}^d \rho_{\theta,l} e_{n-l}(\rho)^{q^l}.$$
Set $u=\frac{\deg(\theta)}{d_\infty}=-v_\infty(\theta)\geq 1.$
We get :
 $$\frac{v_\infty(e_n(\rho))}{q^n}\geq u+{\inf}\{ \frac{v_\infty(e_{n-j}(\rho))}{q^{n-j}} +\frac{v_\infty( \rho_{\theta,j})}{q^n}, j=1, \ldots ,d\}.$$
Let $\beta \in ]0; u[.$ There exists an integer $n_0$ such that:
 $$\forall n\geq n_0,\quad {\inf}\{ \frac{v_\infty( \rho_{\theta,j})}{q^n}, j=1, \ldots ,d\}\geq \beta-u.$$
Therefore:
 $$\forall n\geq n_0, \quad \frac{v_\infty(e_n(\rho))}{q^n}\geq \beta+{\inf}\{ \frac{v_\infty(e_{n-j}(\rho))}{q^{n-j}} , j=1, \ldots ,d\}.$$
Thus, for $n\in [n_0; n_0+d-1],$ we get:
 $$\frac{v_\infty(e_n(\rho))}{q^n}\geq \beta+{\inf}\{ \frac{v_\infty(e_{n_0-j}(\rho))}{q^{n_0-j}} , j=1, \ldots ,d\}.$$
Set:
 $$C={\inf}\{ \frac{v_\infty(e_{n_0-j}(\rho))}{q^{n_0-j}} , j=1, \ldots ,d\}.$$
By induction, we show that if $n\geq n_0+md, m\in \mathbb N,$ then:
$$\frac{v_\infty(e_n(\rho))}{q^n}\geq \beta (m+1)+C.$$
Therefore there exist $n_1\geq n_0,$ $C', C\in \mathbb Q, $ with $C'>0,$ such that:
$$\forall n\geq n_1,\quad  v_\infty(e_n(\rho))\geq C'nq^n +C.$$
\end{proof}

Let $E/K$ be a finite abelian extension $H\subset E.$ Let $G={\rm Gal}(E/K).$ We denote by $S_E$ the set of maximal ideals $P$ of $A$ which are wildly ramified in $E/K$ (note that we can have $S_E=\emptyset$). Let $P$ be a maximal ideal of $A$ such that $P\not \in S_E.$ We fix a maximal ideal $\frak P$ of $O_E$ above $P$. Let $D_P\subset G$ be the decomposition  group associated to $P,$ i.e. $D_P=\{ g\in G, g(\frak P)=\frak P\}.$ We have a natural surjective homomorphism  $D_P \twoheadrightarrow {\rm Gal}(\frac{O_E}{\frak P}/\frac{A}{P}), g\mapsto \bar g.$  Let $I_P$ be the inertia group at $P,$ i.e. $I_P={\ker}(D_P\rightarrow {\rm Gal}(\frac{O_E}{\frak P}/\frac{A}{P})).$ Then, since $P\not \in S_E,$ we have:
 $$\mid I_P\mid \not \equiv 0\pmod{p}.$$
Let ${\rm Frob}_P\in {\rm Gal}(\frac{O_E}{\frak P}/\frac{A}{P})$ be the Frobenius at $P,$ i.e.
 $$\forall x\in  \frac{O_E}{\frak P}, \quad {\rm Frob}_P(x)=x^{q^{\deg P}}.$$
We set:
 $$\sigma_{P,O_E}:=\frac{1}{\mid I_P\mid}\sum_{g\in D_P, \bar g={\rm Frob}_P} g\in \mathbb F_p[G].$$
If $P\in S_E,$ we set:
 $$\sigma_{P,O_E}=0.$$
Note that, if $L/K$ is a finite abelian extension, $L\subset E,$ and if $P$ is unramified  in $L$ with $P\not \in S_E,$ then:
 $$\sigma_{P,O_E}\mid_L=(P, L/K).$$
If $I\in \mathcal I(A), I\subset A,$ $I=\prod_PP^{m_P},$ we set:
 $$\sigma_{I,O_E}=\prod_P\sigma_{P,O_E}^{m_P}\in \mathbb F_p[G].$$
For all $n\in \mathbb Z,$ we set:
 $$\mathcal L(\phi/O_E; n; z)= \sum_{d \geq 0} \sum_{\substack{I\in \mathcal I(A), I\subset A,\\ \deg I=d}} \frac{z^d}{\psi(I)^n} \sigma_{I,O_E}\in H[G][[z]].$$
By the proof of Lemma \ref{LemmaS6.4}, we have:
$$\forall n\in \mathbb Z, \quad \mathcal L(\phi/O_E; n; z)\in \mathbb T_z(H_\infty)[G],$$
and:
$$\mathcal L(\phi/O_E; 1; z)=\prod_{ P}(1-\frac{z^{\deg P}}{\psi(P)}\sigma_{P,O_E})^{-1}\in (\mathbb T_z(H_\infty)[G])^\times .$$
Note that, if $L/K$ is a finite abelian extension, $H\subset L\subset E,$ we have:
$$\mathcal L(\phi/O_E; 1; z)\mid_{\mathbb T_z(L_\infty)}=(\prod_{P\in S_E\setminus S_L}(1-\frac{z^{\deg P}}{\psi(P)}\sigma_{P,O_L})\, \mathcal L(\phi/O_L; 1; z))\mid_{\mathbb T_z(L_\infty)}.$$
We set:
$$I(O_E)=\prod_{P\in S_E}P.$$

Recall that
 $$U(\widetilde{\phi}/O_E[z])= \{ f\in \mathbb T_z(E_\infty), \exp_{\widetilde{\phi}}(f)\in O_E[z]\}.$$

\begin{theorem}\label{TheoremS6.2}
We always have:
 $$\psi(I(O_E))\mathcal L(\phi/O_E; 1; z)O_E[z]\subset U(\widetilde{\phi}/O_E[z]).$$
Furthermore, if $S_E=\emptyset$, we have an equality:
 $$\mathcal L(\phi/O_E; 1; z)O_E[z] = U(\widetilde{\phi}/O_E[z]).$$
\end{theorem}

\begin{proof}
We divide the proof into several steps.

\medskip

\noindent 1) We will first work in $E[[z]].$ Observe that $\exp_{\widetilde{\phi}}:E[[z]]\rightarrow \widetilde{\phi}(E[[z]])$ is an isomorphism of $A$-modules. In fact, if we write: $\log_\phi=\sum_{i\geq 0} l_i(\phi) \tau ^i$, then we set:
 $$\log_{\widetilde{\phi}}=\sum_{i\geq 0} l_i(\phi) z ^i \tau^i.$$
Thus, $\log_{\widetilde{\phi}}$ converges on $E[[z]],$ and $\log_{\widetilde{\phi}}\exp_{\widetilde{\phi}}=\exp_{\widetilde{\phi}}\log_{\widetilde{\phi}}=1.$

\medskip

\noindent 2) Let $P$ be a maximal ideal of $A.$ Let $R_P=S^{-1}O_E\subset E,$ where $S=A\setminus P.$ Then:
$$PR_P=\psi(P)R_P.$$
By Lemma \ref{LemmaS6.5}, we have:
$$\exp_{\widetilde{\phi}}(PR_P[[z]])\subset PR_P[[z]],$$
$$\log_{\widetilde{\phi}}(PR_P[[z]])\subset PR_P[[z]].$$
Thus:
\begin{equation} \label{equation1}
\exp_{\widetilde{\phi}}(PR_P[[z]])=PR_P[[z]].
\end{equation}

\medskip

\noindent 3) Recall that there exists a sign-normalized rank one Drinfeld module $\varphi:=P*\phi:A\hookrightarrow B\{\tau\}$ such that:
$$\forall a\in A, \quad \phi_P\phi_a=\varphi_a\phi_P.$$
Furthermore (\cite{GOS}, Theorem 7.4.8):
$$\forall a\in A, \quad \varphi_a= \sigma_P (\phi_a):=\sum_{i=0}^{r\deg a}\sigma_P(\phi_{a,i})\tau^i.$$
Thus:
$$\exp_{\varphi}=\sigma_P(\exp_\phi):= \sum_{i\geq 0} \sigma_P(e_i(\phi))\tau ^i,$$
$$\log_{\varphi}=\sigma_P(\log_\phi):= \sum_{i\geq 0} \sigma_P(l_i(\phi))\tau ^i.$$
In particular:
$$\phi_P\exp_\phi = \sigma_P(\exp_\phi)\psi(P),$$
$$\psi(P)\log_\phi=\sigma_P(\log_\phi)\phi_P.$$
The same properties hold for $\widetilde{\phi}.$

\medskip

\noindent 4)  Let's set:
$$U(\widetilde{\phi}/R_P[[z]])=\{ x\in E[[z]]; \exp_{\widetilde{\phi}}(x)\in R_P[[z]]\}.$$
Let's assume that $P\not \in S_E.$ Then, by 1) and 2),  $\exp_{\widetilde{\phi}}$ induces an isomorphism of $A$-modules:
$$\frac{E[[z]]}{PR_P[[z]]}\simeq \widetilde{\phi}(\frac{E[[z]]}{PR_P[[z]]}).$$
Therefore, we get an isomorphism of $A$-modules:
$$\frac{U(\widetilde{\phi}/R_P[[z]])}{PR_P[[z]]}\simeq \widetilde{\phi}(\frac{R_P[[z]]}{PR_P[[z]]}).$$
Now observe that:
$$(\widetilde{\phi}_P-z^{\deg P} \sigma _{P,O_E})\widetilde{\phi}(\frac{R_{P}[[z]]}{PR_{P}[[z]]})=\{0\}.$$
Furthermore, if $x\in E[[z]]\setminus R_{P}[[z]],$ then one can easily verify that:
$$(\widetilde{\phi}_P-z^{\deg P} \sigma _{P,O_E})(x)\not \in PR_{P}[[z]].$$
Thus:
$$\widetilde{\phi}(\frac{R_P[[z]]}{PR_P[[z]]})=\{ x\in \widetilde{\phi}(\frac{E[[z]]}{PR_P[[z]]}),(\widetilde{\phi}_P-z^{\deg P} \sigma _{P,O_E})(x)=0\}.$$
Let $x\in E[[z]],$ we deduce that:
$$x\in U(\widetilde{\phi}/R_P[[z]])\Leftrightarrow (\widetilde{\phi}_P-z^{\deg P} \sigma _{P,O_E})(\exp_{\widetilde{\phi}}(x))\in PR_P[[z]].$$
Observe that, by 3), we have:
$$\sigma_{P,O_E}(\exp_{\widetilde{\phi}})=\exp_{\widetilde{\varphi}},$$
and also:
$$\widetilde{\phi}_P\exp_{\widetilde{\phi}}=\sigma_{P,O_E}(\exp_{\widetilde{\phi}}) \psi(P).$$

Thus:
$$x\in U(\widetilde{\phi}/R_{P}[[z]])\Leftrightarrow \exp_{\widetilde{\varphi}}(\psi(P)x-z^{\deg P} \sigma_{P,O_E}(x)) \in PR_{P}[[z]].$$
Applying \eqref{equation1} for $\varphi,$ we have:
$$x\in U(\widetilde{\phi}/R_{P}[[z]])\Leftrightarrow \psi(P)x-z^{\deg P} \sigma_{P,O_E}(x)\in PR_{P}[[z]].$$
Thus:
$$U(\widetilde{\phi}/R_{P}[[z]])= (1-\frac{z^{\deg P}}{\psi(P)}\sigma_{P, O_E})^{-1} R_P[[z]].$$ ${}$

\medskip

\noindent 5) Let $P$ be a maximal ideal of $A.$ If $P\not \in S_E,$ by 4), we have:
 $$U(\widetilde{\phi}/R_P[[z]])=\psi(I(O_E))\mathcal L(\phi/O_E; 1; z)R_P[[z]]= (1-\frac{z^{\deg P}}{\psi(P)} \sigma_{P,O_E})^{-1}R_{P}[[z]].$$
If $P\in S_E,$  then:
 $$\psi(I(O_E))\mathcal L(\phi/O_E; 1; z)R_{P}[[z]]=PR_{P}[[z]]\subset U(\widetilde{\phi}/R_P[[z]]) .$$
Since $\psi(I(O_E))\mathcal L(\phi/O_E;1;z)\in \mathbb  T_z(H_\infty)[G],$ we get:
 $$\psi(I(O_E))\mathcal L(\phi/O_E;1;z)R[z]\subset \mathbb T_z(E_\infty).$$
Observe that $O_E[[z]]=\bigcap_PR_{P}[[z]].$ Therefore, we get:
 $$\exp_{\widetilde{\phi}}(\psi(I(O_E))\mathcal L(\phi/O_E;1;z)O_E[z]) \subset O_E[[z]] \cap \mathbb T_z(E_\infty)=O_E[z].$$
Thus, we get the first assertion.

\bigskip

Now, let's assume that $S_E=\emptyset.$ We have:
$$\bigcap_PU(\widetilde{\phi}/R_{P}[[z]])= \{ x\in E_\infty[[z]], \exp_{\widetilde{\phi}}(x)\in O_E[[z]]\}.$$
By 4), we get:
$$\prod_P(1-\frac{z^{\deg P}}{\psi (P)}\sigma _{P,O_E})\{ x\in E_\infty[[z]], \exp_{\widetilde{\phi}}(x)\in O_E[[z]]\}=O_E[[z]].$$
Thus:
$$\{ x\in E_\infty[[z]], \exp_{\widetilde{\phi}}(x)\in O_E[[z]]\}=\mathcal L(\phi/O_E; 1; z)O_E[[z]].$$
Hence:
$$U(\widetilde{\phi}/R[z])=\mathcal L(\phi/O_E; 1; z)O_E[[z]]\cap \mathbb T_z(E_\infty).$$
Since $ \mathcal L(\phi/O_E; 1; z)\in(\mathbb  T_z(H_\infty)[G])^\times,$ we have:
 $$\mathcal L(\phi/O_E; 1; z)O_E[[z]]\cap \mathbb T_z(E_\infty)= \mathcal L(\phi/O_E; 1; z)O_E[z].$$
This concludes the proof of the theorem.
\end{proof}

%%%%%%%%%%%%%%%%%%%%%%%%%%%%%%%%%%%%%%%%%%%%%%%%%%%%%%%

\subsection{A class formula \`a la Taelman} \label{class formula}${}$

\medskip

Recall that $\ev:\mathbb T_z(E_\infty)\rightarrow E_\infty$ is the evaluation at $z=1.$

\begin{definition} \label{equivariant harmonic series}
We define {\it the equivariant $A$-harmonic series} $\mathcal L(\phi/O_E)$ attached to $\phi/O_E$ by:
 $$\mathcal L(\phi/O_E)= \ev(\mathcal L(\phi/O_E; 1; z))\in (H_\infty[G])^\times.$$
\end{definition}

Note that:
$$\mathcal L(\phi/O_E)=\prod_{ P}(1-\frac{1}{\psi(P)}\sigma_{P,O_E})^{-1}=\sum_{I\in \mathcal I(A), I\subset A}\frac{1}{\psi(I)}\sigma_{I,O_E}.$$

\begin{theorem}\label{TheoremS6.3} We have:
$$\alpha_A(\phi/O_E)=1,$$
i.e.
$$\zeta_{O_E}(1)= [O_E:U(\phi/O_E)]_A[H(\phi/O_E)]_A.$$
Furthermore:
$$\psi(I(O_E))\mathcal L(\phi/O_E)O_E\subset U_{St}(\phi/O_E),$$
and
$$[\frac{U_{St}(\phi/O_E)}{\psi(I(O_E))\mathcal L(\phi/O_E)O_E}]_A=[\phi(\frac{O_E}{I(O_E)O_E})]_A.$$
\end{theorem}

\begin{proof}${}$\par
 \noindent 1)  Let $J\subset I(O_E)$ be a finite product of maximal ideals of $A.$ Set:
$$\mathcal L_J(\phi/O_E):=\prod_{P} (1-\frac{1}{\psi(P)}\sigma_{P,O_E})^{-1}\in (H_\infty[G])^\times,$$
$$\mathcal L_J(\phi/O_E; 1; z):= \prod_{P} (1-\frac{z^{\deg P}}{\psi(P)}\sigma_{P,O_E})^{-1}\in (\mathbb T_z(H_\infty)[G])^\times,$$
where $P$ runs through the maximal ideals of $A$ that do not divide $J.$

By Lemma \ref{LemmaS6.5} and the proof of  Theorem \ref{TheoremS6.2}, we have:
 $$\{ x\in E[[z]], \exp_{\widetilde{\phi}}(x)\in \psi(J)O_E[[z]]\}= \psi(J)\mathcal L_J(\phi/O_E; 1; z)O_E[[z]].$$
We can conclude as in the proof of Theorem \ref{TheoremS6.2} that:
 $$\psi(J)\mathcal L_J(\phi/O_E; 1;z)O_E[z]=\{ x\in \mathbb T_z(E_\infty), \exp_{\widetilde{\phi}}(x)\in \psi(J)O_E[z]\}.$$
Therefore, we have a short  exact sequence of $A$-modules:
\begin{align} \label{suite exacte 1}
0 \rightarrow & \frac{U(\widetilde{\phi}/O_E[z])}{\psi(J)\mathcal L_J(\phi/O_E; 1; z)O_E[z]}\rightarrow  \widetilde{\phi}(\frac{O_E[z]}{\psi(J)O_E[z]}) \rightarrow\\
& \rightarrow \frac{\mathbb T_z(E_\infty)}{\psi(J)O_E[z]+\exp_{\widetilde{\phi}}(\mathbb T_z(E_\infty))}\rightarrow H(\widetilde{\phi}/O_E[z])\rightarrow 0. \notag
\end{align}
 Note that $ \widetilde{\phi}(\frac{O_E[z]}{\psi(J)O_E[z]})$ is a finitely generated and free $\mathbb F_q[z]$-module. Let $\rho$ be the Drinfeld module defined over $O_E$ such that:
$$\exp_\rho= \psi(J)^{-1}\exp_\phi \psi(J).$$
Then, the map $x\mapsto \psi(J)^{-1}x$ induces   an isomorphism of $A$-modules (the left module is an $A$-module via $\phi$ and the right module is an $A$-module via $\rho$):
$$\frac{\mathbb T_z(E_\infty)}{\psi(J)O_E[z]+\exp_{\widetilde{\phi}}(\mathbb T_z(E_\infty))}\simeq H(\widetilde{\rho}/O_E[z]).$$
Observe that ${\ker} \, \ev = (z-1)\mathbb T_z(E_\infty).$ Furthermore, since $O_E[z]\cap(z-1)\mathbb T_z(E_\infty)=(z-1)O_E[z],$ we have :
$$U(\widetilde{\phi}/O_E[z])\cap {\ker} \, \ev =(z-1)U(\widetilde{\phi}/O_E[z]),$$
$$\psi(J)\mathcal L_J(\phi/O_E; 1; z)O_E[z]\cap {\ker}\, \ev= (z-1)\psi(J)\mathcal L_J(\phi/O_E; 1; z)O_E[z].$$
Thus, the evaluation at $z=1$ induces the following exact sequence of $A$-modules:
\begin{equation} \label{suite exacte 2}
0\rightarrow (z-1)\frac{U(\widetilde{\phi}/O_E[z])}{\psi(J)\mathcal L_J(\phi/O_E;1;z)O_E[z]} \rightarrow \frac{U(\widetilde{\phi}/O_E[z])}{\psi(J)\mathcal L_J(\phi/O_E;1;z)O_E[z]} \rightarrow\frac{U_{St}(\phi/O_E)}{\psi(J)\mathcal L_J(\phi/O_E)O_E} \rightarrow 0.
\end{equation}

Note also that the evaluation at $z=1$ induces a sequence of $A$-modules:
\begin{equation} \label{suite exacte 3}
0\rightarrow (z-1)\widetilde{\phi}(\frac{O_E[z]}{\psi(J)O_E[z]})\rightarrow \widetilde{\phi}(\frac{O_E[z]}{\psi(J)O_E[z]})\rightarrow \phi(\frac{O_E}{\psi(J)O_E}) \rightarrow  0.
\end{equation}
For an $\mathbb F_q[z]$-module $M$, we denote by $M[z-1]$ the $(z-1)$-torsion.
By \eqref{suite exacte 1}, \eqref{suite exacte 2}, \eqref{suite exacte 3} and  the Snake Lemma, we get the following exact sequence of finite $A$-modules:
\begin{align*}
0 \rightarrow  H(\widetilde{\rho}/O_E[z])[z-1] & \rightarrow H(\widetilde{\phi}/O_E[z])[z-1]\rightarrow \frac{U_{St}(\phi/O_E)}{\psi(J)\mathcal L_J(\phi/O_E)O_E} \rightarrow \\
& \rightarrow  \phi(\frac{O_E}{\psi(J)O_E}) \rightarrow H(\rho/O_E)\rightarrow H(\phi/O_E)\rightarrow 0.
\end{align*}

By the proof of Theorem \ref{TheoremS4.1}, we have:
$$[H(\widetilde{\rho}/O_E[z])[z-1]]_A= [H(\rho/O_E)]_A,$$
$$[H(\widetilde{\phi}/O_E[z])[z-1]]_A=[H(\phi/O_E)]_A.$$
Thus:
$$[\frac{U_{St}(\phi/O_E)}{\psi(J)\mathcal L_J(\phi/O_E)O_E}]_A=[\phi ( \frac{O_E}{JO_E})]_A.$$

\medskip

\noindent 2) Now, we have:
$$[O_E: \mathcal L_J(\phi/O_E)O_E]_A=\frac{{\det}_{K_\infty}\mathcal L_J(\phi/O_E)}{\sgn({\det}_{K_\infty}\mathcal L_J(\phi/O_E))}.$$
Thus:
$$[O_E: \psi(J)\mathcal L_J(\phi/O_E)O_E]_A=[\frac{O_E}{JO_E}]_A\frac{{\det}_{K_\infty}\mathcal L_J(\phi/O_E)}{\sgn({\det}_{K_\infty}\mathcal L_J(\phi/O_E))}.$$
And finally, we get:
$$[O_E:U_{St}(\phi/O_E)]_A= \frac{[\frac{O_E}{JO_E}]_A}{[\phi (\frac{O_E}{JO_E})]_A}\frac{{\det}_{K_\infty}\mathcal L_J(\phi/O_E)}{\sgn({\det}_{K_\infty}\mathcal L_J(\phi/O_E))}.$$
Set:
$$L_J=\prod_{\frak P\mid J}\frac{[\frac{O_E}{\frak P}]_A}{[\phi (\frac{O_E}{\frak P})]_A}.$$
Then:
$$[O_E:U_{St}(\phi/O_E)]_A=L_J \frac{{\det}_{K_\infty}\mathcal L_J(\phi/O_E)}{\sgn({\det}_{K_\infty}\mathcal L_J(\phi/O_E))}.$$

\medskip

\noindent 3) Let $N\geq 1,$ and we define  $J_N$ to be the l.c.m. of the product of all maximal  ideals of degree $\leq N$ and $I(O_E).$ We have:
$$\lim_{N \rightarrow +\infty } L_{J_N}=L_A(\phi/O_E),$$
$$\lim_{N \rightarrow +\infty } \mathcal L_{J_N}(\phi/O_E)=1.$$
In particular:
$$\lim_{N \rightarrow +\infty } {\det}_{K_\infty} \mathcal L_{J_N}(\phi/O_E)=1.$$
Thus:
$$[O_E:U_{St}(\phi/O_E)]_A=L_A(\phi/O_E).$$
If we apply Theorem \ref{TheoremS4.1} and Proposition \ref{PropositionS6.2}, we get:
$$\zeta_{O_E}(1)=[O_E: U(\phi/O_E)]_A[H(\phi/O_E)]_A.$$
\end{proof}

%%%%%%%%%%%%%%%%%%%%%%%%%%%%%

%\section{Anderson's Log-Algebraicity Theorem revisited}\label{And}${}$\par

\section{Log-Algebraicity Theorem}\label{And}

\subsection{A refinement of Anderson's log-algebraicity theorem}${}$

\medskip

We keep the notation of the previous sections.

\begin{lemma}\label{LemmaS6.6}
Let $E/K$ be a finite separable extension, $H\subset E.$ Let $P$ be a maximal ideal of $A$ which is unramified in $E.$  Let $\lambda_P\in \overline{K}\setminus\{0\}$ be a root of $\phi_P.$ Then:
$$O_{E(\lambda_P)}=O_E[\lambda_P].$$
\end{lemma}
\begin{proof}
Let $F=E(\lambda_P).$ Recall that $F/E$ is a finite abelian extension unramified outside $P,\infty,$ and totally ramified at $P$ (\cite{GOS}, Proposition 7.5.18).  We also have:
$$[F:E]=q^{\deg P}-1.$$
Write: $\phi_P=\sum_{k=0}^{\deg P} \phi_{P,k} \tau^k,$ $\phi_{P,k} \in B \subset O_E.$ Recall that $\phi_{P,0}=\psi(P)$ and $\phi_{P,\deg P}=1.$ Furthermore, $P$ is unramified in $E/K$ and:
$$\psi(P)O_E=PO_E.$$
Let:
$$G(X)=\sum_{k=0}^{\deg P} \phi_{P,k} X^{q^k-1}\in O_E[X].$$
Then, for any maximal ideal $\frak P$ of $O_E$ above $P:$
$$G(X)\equiv X^{q^{\deg P}-1}\pmod{\frak P}.$$
This implies that $G(X)$ is an Eisenstein polynomial at $\frak P$ for every maximal ideal of $O_E$  $\frak P$ above $P.$ Furthermore:
$$ XG'(X)+G(X)=\psi(P).$$
Therefore:
$$N_{F/E}(G'(\lambda_P))O_E= P^{q^{\deg P}-2}O_E.$$
But $P^{q^{\deg P}-2}O_E$ is the discriminant of $O_F/O_E.$ Thus $O_F=O_E[\lambda_P].$
\end{proof}

Let $E/K$ be a finite abelian extension, $H\subset E.$ Let $n\geq 0$ be an integer, let $X_1, \ldots, X_n$ be $n$ indeterminates over K. Let  $\tau: E[X_1, \ldots, X_n][[z]]\rightarrow E[X_1, \ldots, X_n][[z]]$ be the $\mathbb F_q[[z]]$-homomorphism continuous for the $z$-adic topology such that:
 $$\forall f\in E[X_1, \ldots, X_n], \quad \tau (f)=f^q.$$
For a non-zero ideal $I$ of $A$ and for $f=\sum_{i_1, \ldots, i_n\in \mathbb N}f_{i_1, \ldots, i_n} X_1^{i_1}\cdots X_n^{i_n}\in E[X_1, \ldots, X_n],$ with $f_{i_1, \ldots, i_n}\in E,$ we set:
$$I*_Ef=  \sum_{i_1, \ldots, i_n\in \mathbb N}\sigma_{I,O_E}(f_{i_1, \ldots, i_n}) \phi_I(X_1)^{i_1}\cdots \phi_I(X_n)^{i_n},$$
where $\sigma_{I, O_E}$ is defined in Section \ref{CF}. Recall that $I(O_E)$ is the product of maximal ideals of $A$ that are wildly ramified in $E/K.$

\begin{theorem}\label{TheoremS6.4}
For all $f\in O_E[X_1, \ldots, X_n],$ we have:
$$\exp_{\widetilde{\phi}}(\psi(I(O_E))\sum_{I\in \mathcal I(A), I\subset A}\frac{I*_Ef}{\psi(I)} z^{\deg I})\in O_E[X_1, \ldots, X_n, z].$$
In particular, for all $f\in B[X_1, \ldots, X_n],$ we have:
$$\exp_{\widetilde{\phi}}(\sum_{I\in \mathcal I(A), I\subset A}\frac{I*_Hf}{\psi(I)} z^{\deg I})\in B[X_1, \ldots, X_n, z].$$
\end{theorem}

\begin{remark}
This result is a  generalization of the Log-Algebraicity Theorems established in \cite{AND}, \cite{AND2} (in these papers the theorem is proved for $E=H,$  $d_\infty=1$ and $n\leq 1$). Furthermore, the result in the case $E=H$ can be proved along the same lines as that used to prove \cite{AND2}, Theorem 3. Following \cite{ATR}, Section 2.6,  we will show below how Theorem \ref{TheoremS6.2} implies the Log-Algebraicity Theorem. Observe also that the case $n=0$ is a  direct consequence of Theorem \ref{TheoremS6.2}.
\end{remark}

\begin{proof}
Let's write:
$$\exp_{\widetilde{\phi}}(\psi(I(O_E))\sum_I\frac{I*_Ef}{\psi(I)} z^{\deg I})=\sum_{m\geq 0} g_m(X_1, \ldots, X_n)z^m, $$
with $g_m(X_1, \ldots, X_n)\in E[X_1, \ldots, X_n].$

\medskip

\noindent 1) Let $P_1, \ldots, P_n$ be $n$ distinct maximal ideals of $A$ which are unramified in $E,$  with $q^{\deg P_i}\geq 3, i=1, \ldots, n,$ and for $i=1, \ldots, n,$ let $\lambda_i\not =0$ be a root of $\phi_{P_i}.$ Set:
$$F=E(\lambda_1, \ldots, \lambda_n).$$
Then $F/E$ is unramified outside $P_1, \ldots, P_n, \infty,$ $F/K$ a finite abelian extension of $K$ which is tamely ramified at $P_1, \ldots, P_n.$ Let $O_F$ be the integral closure of $A$ in $F.$  Let $Q$ be any maximal ideal of $A,$ if $Q$ is not wildly ramified in $E,$ we have (\cite{GOS}, Proposition 7.5.4):
 $$\sigma_{Q,O_F}(\lambda_i)=\phi_Q(\lambda_i), \quad \text{if} \quad Q\not =P_i,$$
and:
 $$\sigma_{P_i,O_F}(\lambda_i)=0.$$
We deduce that:
 $$\psi(I(O_E))\sum_I\frac{I*_Ef}{\psi(I)} z^{\deg I}\mid_{X_i=\lambda_i}= \psi(I(O_F))\mathcal L(\phi/O_F; 1;z)f(\lambda_1, \ldots, \lambda_n).$$
Therefore, by Theorem \ref{TheoremS6.2}, we get:
 $$\forall m\geq 0, \quad g_m(\lambda_1, \ldots, \lambda_n)\in O_F.$$
Let $i\in \{1, \ldots, n\}.$ Then:
 $$E(\lambda_i)\cap E(\lambda_1, \ldots \lambda_{i-1}, \lambda_{i+1}, \ldots, \lambda_n)=E.$$
Furthermore, the discriminant of $O_{E(\lambda_i)}/O_E$ and $O_{E(\lambda_1, \ldots \lambda_{i-1}, \lambda_{i+1}, \ldots, \lambda_n)}/O_E$ are relatively prime, thus, by Lemma \ref{LemmaS6.6}, we have:
 $$O_F=O_E[\lambda_1, \ldots, \lambda_n].$$
 Finally, for $m\geq 0,$ for $n$ distinct maximal ideals $P_1, \ldots, P_n$ of $A$ that are unramified in $E/K,$ with $q^{\deg P_i}\geq 3,i=1, \ldots, n,$ and for $i=1, \ldots, n,$ if $\lambda_i\not =0$ be a root of $\phi_{P_i},$ then we have:
 $$g_m(\lambda_1, \ldots, \lambda_n)\in O_E[\lambda_1, \ldots, \lambda_n].$$
This implies:
$$\forall m\geq 0, \quad g_m(X_1, \ldots, X_n)\in O_E[X_1, \ldots, X_n].$$

\medskip

\noindent 2)  We fix a $K$-embedding of $\overline{K}$ in $\mathbb C_\infty.$
For $\sigma \in {\rm Gal}(H/K),$ let $\Lambda(\phi^{\sigma})\subset \mathbb C_\infty$ be the $A$-module of periods of $\phi^{\sigma},$ and let  $\Lambda(\phi^{\sigma})K_\infty$ be the $K_\infty$-vector space generated by $\Lambda(\phi^{\sigma}).$  Then $\frac{\Lambda(\phi^{\sigma})K_\infty}{\Lambda(\phi^{\sigma})}$ is compact, thus there exists a constant $C\in \mathbb R$ such that:
$$\forall \sigma\in {\rm Gal}(H/K), \forall x\in \Lambda(\phi^{\sigma})K_\infty, \quad v_\infty(\exp_{\phi^{\sigma}}(x))\geq C.$$
Recall that, if $\sigma \in {\rm Gal}(H/K),$  then there exists a non-zero ideal $J$ of $A$  such that $\sigma=(J,H/K)=\sigma_J,$ and we have (\cite{GOS}, Theorem 7.4.8):
$$\phi_J \phi_a= \phi^{\sigma}_a\phi_J.$$
Thus:
$$\exp_{\phi^{\sigma}}\psi(J)=\phi_J \exp_\phi.$$
In particular:
$$\Lambda(\phi^{\sigma})= \psi(J)J^{-1}\Lambda(\phi),$$
$$\Lambda(\phi^{\sigma})K_\infty= \psi(J)\Lambda(\phi)K_\infty.$$
Therefore, there exists a constant $C'\in \mathbb R,$ such that:
$$\forall \sigma \in {\rm Gal}(H/K), \forall x_1, \ldots, x_n \in \Lambda(\phi^{\sigma})K_\infty, \forall I \in \mathcal I(A), \quad v_\infty(I *^\sigma_E f^{\sigma}\mid_{X_i=\exp_{\phi^{\sigma}}(x_i)})\geq C',$$
where $*^\sigma_E$ is the map $*$ attached to $\phi^{\sigma}.$
Now, recall that $\exp_\phi =\sum_{j\geq 0} e_j(\phi)\tau^j,$ then there exists a constant $C''>0$ such that (Lemma \ref{LemmaS6.5}):
$$\forall \sigma\in {\rm Gal}(H/K),\forall j \gg 0, \quad v_\infty(e_j(\phi^{\sigma}))\geq C''jq^j.$$
Note also that there exists $C'''\in \mathbb R$ such that:
$$\forall \sigma \in {\rm Gal}(H/K), \forall I \in \mathcal I(A), \deg I=m \gg 0,  \quad v_\infty(\frac{1}{\sigma(\psi(I))})\geq \frac{m}{d_\infty}+C'''.$$
This implies that there exists an integer $m_0\in \mathbb N,$ such that:
$$\forall m\geq m_0,\forall \sigma\in {\rm Gal}(E/K),  \forall \lambda_1, \ldots, \lambda_n\in \exp_{\phi^{\sigma}}(\Lambda(\phi^{\sigma})K_\infty), \quad v_\infty(g_m^{\sigma}(\lambda_1, \ldots, \lambda_n))>0.$$

\medskip

\noindent 3) Let $m_0\in \mathbb N$ be as in 2). Let $\lambda_1, \ldots, \lambda_n$ be $n$ torsion points for $\phi.$ Let $F=E(\lambda_1, \ldots, \lambda_n).$ Then $F/K$ is a finite abelian extension.  Let $w$ be a place of $F$ above $\infty.$  Let $i_w:E\rightarrow \mathbb C_\infty$ be the $K$-embedding of $F$ in $\mathbb C_\infty$ corresponding to $w.$ Then there exists $\sigma\in {\rm Gal}(F/K)$ such that:
$$\forall m\geq 0, \quad i_w(g_m(\lambda_1, \ldots, \lambda_n))= \sigma(g_m(\lambda_1, \ldots, \lambda_n))=g_m^{\sigma}(\sigma(\lambda_1),\ldots, \sigma(\lambda_n)).$$
Observe that $\sigma(\lambda_i)\in \exp_{\phi^{\sigma}}(\Lambda(\phi^{\sigma})K_\infty), i=1, \ldots, n$ (\cite{GOS}, Proposition 7.5.16).
Therefore:
$$\forall m\geq m_0, \quad w(g_m(\lambda_1, \ldots, \lambda_n))>0.$$
Thus, we get that for any place $w$ of $F$ above $\infty:$
$$\forall m\geq m_0, \quad w(g_m(\lambda_1, \ldots, \lambda_n))>0.$$
But by 1), $\forall m\geq 0,$ $g_m(\lambda_1, \ldots, \lambda_n)\in O_F.$ Since $O_F$ is the set of elements of $F$ which are regular outside the places of $F$ above $\infty,$ we deduce that:
$$\forall m\geq m_0, \quad g_m(\lambda_1, \ldots, \lambda_n)=0.$$
And the above property is true for any $n$ torsion points of $\phi,$ thus:
$$\forall m\geq m_0, \quad g_m(X_1, \ldots, X_n)=0.$$
\end{proof}

M. Papanikolas informed us that, together with N. Green, they obtained  explicit formulas for Anderson's Log-Algebraicity Theorem (\cite{AND}, Theorem 5.1.1) when the genus $g$ of $K$ is one and $d_\infty=1$.

%%%%%%%%%%%%%%%%%%%%%%

\subsection{Several variable $L$-series and shtukas}\label{Section one variable}${}$

\medskip

In this section, we present an alternative approach to the several variable Log-Algebraicity Theorem (Theorem \ref{TheoremS6.4}) by using the seminal works of Drinfeld \cite{DRI1}, \cite{DRI2}, \cite{DRI3} (see also \cite{AND}, \cite{THA2}, and \cite{GOS}, Chapter 6).

We recall some notation for the convenience of the reader. Let $X/\mathbb F_q$ be a smooth projective geometrically irreducible curve of genus $g$ whose function field is $K$. We will consider $\infty$ as a closed point of $X$ of degree $d_\infty.$ Recall that $K_\infty$ is the completion of $K$ at $\infty,$ $\bar K_\infty$ is a fixed algebraic closure of $K_\infty,$ and  $\mathbb C_\infty$ is the completion of $\bar K_\infty.$  Let $\sgn: K_\infty^\times \rightarrow \mathbb F_\infty^\times$ be a sign function ($\mathbb F_\infty$ is the residue field of $K_\infty$ and $d_\infty=[\mathbb F_\infty:\mathbb F_q]$), i.e. $\sgn$ is a group homomorphism such that $\sgn\mid_{\mathbb F_\infty^\times}={\rm Id}\mid_{\mathbb F_\infty^\times}.$ We fix  $\pi \in K\cap {\rm Ker}(\sgn)$ and  such that $K_\infty =\mathbb F_\infty((\pi)).$

We set $\bar X=X \otimes_{\mathbb F_q} \mathbb C_\infty,$ and $\bar A:=A\otimes_{\mathbb F_q}\mathbb C_\infty.$ Then $F:={\rm Frac}(\bar A)$ is the function field of $\bar X.$ We identify $\mathbb C_\infty$ with its image $1\otimes \mathbb C_\infty $ in $F.$ Note that $\bar A$ is the set of elements of $F/\mathbb C_\infty$ which are ``regular outside $\infty$''. We denote by $\tau: F\rightarrow  F$ the $K$-algebra homomorphism such that:
$$\tau\mid_{\bar A}={\rm Id}_A \otimes {\rm Frob}_{\mathbb C_\infty}.$$
For $m\geq 0,$ we also set:
$$\forall x\in F, \quad x^{(m)}=\tau^m (x).$$
Let $P$ be a point of $\bar X (\mathbb C_\infty).$ We denote by $P^{(i)}$ the point of $\bar X(\mathbb C_\infty)$ obtained by applying $\tau^i$ to  the coordinates of $P.$  If $D\in {\rm Div}(\bar X),$ $D=\sum_{j=1}^n   n_{P_j} (P_j),$ $P_j\in \bar X(\mathbb C_\infty),$  $n_{P_j}\in \mathbb Z ,$ we set:
$$D^{(i)}=\sum_{j=1}^n  n_{P_j} (P_j^{(i)}).$$
If $D=(x)$, $x\in F^\times,$ then:
$$D^{(i)}= (x^{(i)}).$$

We fix a point $\bar \infty$ of $X(\mathbb C_\infty)$ above $\infty.$ Let $\xi$ be the point of $\bar X (\mathbb C_\infty)$ corresponding to the kernel of the map $\bar A\rightarrow \mathbb C_\infty, \sum x_i\otimes a_i\mapsto \sum x_ia_i.$ Let $\rho: K \rightarrow K \otimes 1, x \mapsto x\otimes 1.$ Then:
 $$F=\mathbb C_\infty (\rho(K)).$$
By \cite{THA2} (see also \cite{GOS}, section 7.11), there exists a function $f\in F^\times,$ such that:
$$V^{(1)}-V+(\xi)-(\bar \infty)=(f),$$
for some effective divisor $V$ of $\bar X/\mathbb C_\infty$ of degree $g.$  The points $\xi$ and $\bar \infty^{(-1)}$  do not belong to the support of $V$ (\cite{THA2}, Corollary 0.3.3). We identify the completion of $F$ at $\bar \infty$  with:
 $$\mathbb C_\infty((\frac{1}{t})),$$
where $t=\rho(\pi^{-1}).$ We have a natural sign function  $\overline{\sgn}: \mathbb C_\infty((\frac{1}{t}))^\times \rightarrow \mathbb C_\infty^\times$ attached to $\frac{1}{t}.$ We normalize $f$ such that $\overline{\sgn}(f)=1.$

We set:
 $$(\infty)=\sum_{i=0} ^{d_\infty-1} (\bar \infty^{(i)}),$$
 $$W(\mathbb C_\infty)=\bigcup_{m\geq 0}L(V+m( \infty)),$$
where:
 $$L(V+m( \infty))=\{x\in  F^\times, (x)+V+m( \infty)\geq 0\}\cup \{0\}.$$
%For $md_\infty\geq g-1$, we have:
 %$$\dim_{\mathbb C_\infty} L(V+m( \infty))= md_\infty+1.$$
Observe that for $i>0:$
\begin{equation} \label{divisor}
(ff^{(1)}\cdots f^{(i-1)})= V^{(i)}-V+(\xi)+\cdots +(\xi^{(i-1)})-\sum_{k=0}^{i-1} (\bar \infty^{(k)}).
\end{equation}
We have (see for example \cite{THA2}, paragraph 0.3.5):
$$W(\mathbb C_\infty)=\oplus_{i\geq 0} \mathbb C_\infty f\cdots f^{(i-1)}.$$
If $L$ is a sub-$\mathbb F_q$-algebra of $\mathbb C_\infty,$ we set:
$$W(L)=\oplus_{i\geq 0} L f\cdots f^{(i-1)}.$$

Let  $a\in A,$ then we can write:
$$\rho (a)=a \otimes 1=\sum_{i=0}^{\deg a} \phi_{a,i} f\cdots f^{(i-1)},$$
where $\phi_{a,i} \in \mathbb C_\infty,$ and:
$$\phi_{a,\deg a}=\overline{\sgn} (a),$$
$$\phi_{a,0}=a.$$
In particular, note  that $\bar \infty$ does not belong to the support of $V.$ The map $\phi:A\rightarrow \mathbb  C_\infty\{\tau\}$ such that:
$$\forall a\in A, \quad \phi_a=\sum \phi_{a,i} \tau^i,$$
is a sign-normalized rank one Drinfeld module by the Drinfeld correspondence attached to $f$ (\cite{THA2}, paragraph 0.3.5, see also  \cite{GOS}, section 7.11 ). Let's write:
 $$\exp_\phi=\sum e_i(\phi) \tau ^i, \quad e_i(\phi) \in \mathbb C_\infty.$$
\noindent We have  (\cite{THA2}, Proposition 0.3.6):
  $$\forall i\geq 0, \quad e_i(\phi)=\frac{1}{f\cdots f^{(i-1)}\mid_{\xi^{(i)}}}.$$

Let $\mathbb H= {\rm Frac}(A \otimes B) \subset F.$ By Drinfeld's correspondence (see \cite{GOS}, Chapter 6), $f\in \mathbb H.$ Thus:
 $$f=t+\sum_{i\geq 0} f_i \frac{1}{t^i}\in H((\frac{1}{t}))\subset \mathbb C_\infty((\frac{1}{t})),$$
where $f_i\in H, \forall i\geq 0.$

We view $\mathbb H$ as a function field over $\rho(K)=K \otimes 1.$ Let $\mathbb K={\rm Frac}(A \otimes A).$  Let $\infty$ be the unique place of $\mathbb K/\rho (K)$ which is above the place $\infty$ of $K/\mathbb F_q.$ Then the completion of $\mathbb K$ above $\infty$ is:
 $$\mathbb K_\infty= \rho(K)(\mathbb F_\infty)((1 \otimes \pi)).$$
%We denote by $\sgn: \mathbb K_\infty^\times \rightarrow \rho(K)^\times$ the sign function attached to $1 \otimes \pi.$
Observe that the set of elements of $\mathbb K/\rho(K)$ which are regular outside $\infty$ is:
 $$\mathbb A:= A[\rho(K)]=K \otimes A.$$
We set $\mathbb B:=B[\rho(K)]=K \otimes B,$ then $\mathbb B$ is the integral closure of $\mathbb A$ in $\mathbb H.$ Let $G={\rm Gal}(H/K)\simeq {\rm Gal}(\mathbb H/\mathbb K).$   Let $\varphi:  \mathbb A \rightarrow \mathbb H\{\tau\}$ be the $\rho(K)$-algebra homomorphism such that:
 $$\forall a\in A, \quad \varphi _a=\sum_{i=0}^{\deg a} \phi_{a,i} f\cdots f^{(i-1)}\tau^i \in \mathbb H\{\tau \}.$$
Let $\exp_\varphi \in \mathbb H\{\{\tau\}\}$ be the following element:
 $$\exp_\varphi =\sum_{i\geq 0} f\cdots f^{(i-1)} e_i(\phi) \tau^i=\sum_{i\geq 0} \frac{f\cdots f^{(i-1)}}{f\cdots f^{(i-1)}\mid_{\xi^{(i)}}}\tau^i.$$
Then:
 $$\forall a\in \mathbb A, \quad \exp_\varphi  a=\varphi_a\exp_\varphi.$$
Let $\mathbb H_\infty=\mathbb H\otimes_{\mathbb K}\mathbb K_\infty,$ then $\exp_\varphi$ converges on $\mathbb H_\infty.$

Let $\frak P$ be a maximal ideal of $B.$ Then $\frak P\mathbb B$ is a maximal ideal of $\mathbb B.$ Let $v_{\frak P}: \mathbb H\rightarrow \mathbb Z\cup\{+\infty\}$ be the valuation on $\mathbb H$ attached to $\frak P\mathbb B.$ Since for all $a\in A,$ $\rho(a)=\sum_{j=0}^{\deg a} \phi_{a,j} f\cdots f^{(j-1)},$ we deduce  that:
 $$\forall i\geq 0, \quad v_{\frak P}(f^{(i)})=q^i v_{\frak P}(f)=0.$$
However, we warn the reader that, if $g>0,$ we have:
 $$f\not \in \mathbb B.$$
 We set:
 $$W(B)= \oplus_{i\geq 0} Bf\cdots f^{(i-1)}.$$

\begin{lemma}\label{LemmaS7.1}${}$\par
\noindent 1)  $W(B)$ is a $A \otimes B$-module containing $A \otimes B,$ furthermore $W(B)$ is a $A \otimes A$-module via $\varphi.$\par
\noindent 2) Let $W(B)\mathbb B$ be the $\mathbb B$-module generated by $W(B).$  Let $\frak P$ be a maximal ideal of $B.$ The inclusion $\mathbb B\subset W(B)\mathbb B$ induces an equality:
 $$\frac{\mathbb B}{\frak P\mathbb B}= \frac{W(B)\mathbb B}{\frak PW(B)\mathbb B}.$$
\noindent 3) $W(B)\mathbb B$ is a fractional ideal of $\mathbb B$. In particular, it is discrete in $\mathbb H_\infty$.
\end{lemma}

\begin{proof}
We have:
 $$\forall i\geq 0, \forall a\in A, \quad \rho(a) f\cdots f^{(i-1)}= \sum_{j=0}^{\deg a} \phi_{a,j}^{q^i} f\cdots f^{(i+j-1)}\in W(B).$$
Observe that:
 $$\forall i,j \geq 0, \quad f\cdots f^{(j-1)}\tau^j(f\cdots f^{(i-1)})= f\cdots f^{(i+j-1)}.$$
The assertion 1) follows.

We set: $O_\frak P=\{x\in \mathbb H, v_{\frak P}(x)\geq 0\}$. Since $\frac{O_{\frak P}}{\frak P O_{\frak P}} \simeq \frac{\mathbb B}{\frak P\mathbb B}$ and $\mathbb B\subset W(B)\mathbb B \subset O_{\frak P},$ the assertion 2) holds.

Let's prove the assertion 3). Note that $A \otimes H$ is the set of elements of $\mathbb H$ which are regular outside $\bar \infty.$ By the expression \eqref{divisor} of the divisor of $f\cdots f^{(i-1)}, i\geq 0,$ there exists $a\in A \otimes B \setminus\{0\}$ such that:
 $$\forall i\geq 0, \quad af\cdots f^{(i-1)}\in A \otimes H.$$
Since  for every maximal ideal $\frak P$ of $B,$ and for all $i \geq 0,$  $v_{\frak P}(f\cdots f^{(i-1)})=0,$ we deduce that:
 $$\forall i\geq 0, af\cdots f^{(i-1)}\in A \otimes B.$$
Thus, there exists $a\in \mathbb B\setminus \{0\}$ such that $aW(B)\subset \mathbb B.$ Since $\mathbb B$ is discrete in $\mathbb H_\infty,$ we get the desired result.
 \end{proof}

Let's observe that, by Lemma \ref{LemmaS7.1}, $W(B)\mathbb B$ is an $\mathbb A$-module via $\varphi.$ Let $\frak P$ be a maximal ideal of $B,$ then, again by Lemma \ref{LemmaS7.1}, $\frac{\mathbb B}{\frak P \mathbb B}$ is an $\mathbb A$-module via $\varphi,$  and we denote this latter $\mathbb A$-module by $\varphi(\frac{\mathbb B}{\frak P \mathbb B}).$

\begin{lemma}\label{LemmaS7.2}
Let $\frak P$ be a maximal ideal of $B.$ Then:
 $${\rm Fitt}_{{\mathbb  A}}\varphi(\frac{\mathbb B}{\frak P \mathbb B})=( [\frac{B}{\frak P B}]_A-\rho([\frac{B}{\frak P B}]_A)){\mathbb A}.$$
\end{lemma}

\begin{proof}
Recall that:
 $$[\frac{B}{\frak P B}]_A=\psi(P^e),$$
where $e=\dim_{\frac{A}{P}}\frac{B}{\frak P}.$  Set $aA=P^e$ where $\sgn(a)=1.$ Then:
 $$\rho(a)= \sum \phi_{a,i} f\cdots f^{(i-1)}.$$
Therefore:
 $$\forall x\in \frac{\mathbb B}{\frak P\mathbb B}, \quad \varphi_{a-\rho(a)}(x)=0.$$
Thus, by similar arguments to those of \cite{APT}, Lemma 5.8, we have an ${\mathbb  A}$-module isomorphism:
 $$\varphi(\frac{\mathbb B}{\frak P \mathbb B})\simeq \frac{{\mathbb  A}}{(a-\rho(a)){\mathbb A}}.$$
\end{proof}

If $M$ is an ${\mathbb A}$-module such that $M$ is a finite dimensional $\rho(K)$-vector space and  its Fitting ideal is principal, ${\rm Fitt}_{{\mathbb  A}}(M)=x{\mathbb  A}$, then we set:
$$[M]_{{\mathbb  A}}= \frac{x}{\sgn(x)}.$$
By the above Lemma, we can form the $L$-series attached to $\varphi/W(B):$
 $$L(\varphi/W(B))=\prod_{\frak P}\frac{[\frac{\mathbb B}{\frak P \mathbb B}]_{{\mathbb   A}}}{[\varphi(\frac{\mathbb B}{\frak P \mathbb B})]_{{\mathbb  A}}}=\prod_{\frak P}(1-\frac{\rho([\frac{B}{\frak P B}]_A)}{[\frac{B}{\frak P B}]_A})^{-1} \in {\mathbb K_\infty}^\times.$$
Note that $L(\varphi/W( B))$  is in fact an element in the $\infty$-adic  completion of $K_\infty[\rho(A)]=A \otimes K_\infty$ which is an affinoid algebra over $K_\infty$, and $L(\varphi/W( B))$ is a special value of a twisted zeta function (see \cite{ANDTR}, Section 5.2). \par

We denote by $\tau:\mathbb H_\infty\rightarrow \mathbb H_\infty$ the continuous homomorphism of $\rho(K)$-algebras such that $\forall x\in  H_\infty, \tau (x) =x^q.$ Let $z$ be an indeterminate. The map $\tau: \mathbb H_\infty \rightarrow \mathbb H_\infty$ extends uniquely into a continuous homomorphism (for the $z$-adic topology) of $\mathbb F_q[[z]]$-algebras $\tau: \mathbb H_\infty [[z]]\rightarrow \mathbb H_\infty [[z]].$  Let $\mathbb T_z(\mathbb H_\infty)\subset \mathbb H_\infty[[z]]$ be the $\infty$-adic completion of $\mathbb H_\infty[z],$ i.e. an element $g\in \mathbb T_z(\mathbb H_\infty)$ can be uniquely written $g=\sum_{i\geq 0} g_i z^i, g_i \in \mathbb H_\infty,$ such that $\lim_{i\rightarrow +\infty} g_i=0.$ We also denote by $\mathbb  T_z(\mathbb K_\infty)$ the $\infty$-adic completion of $\mathbb K_\infty[z].$ Note that $\mathbb T_z(\mathbb H_\infty)$ is a free $\mathbb T_z(\mathbb K_\infty)$-module  of rank $[H:K],$ and if $(e_1, \ldots, e_n)$ is a $K$-basis of $H$ ($n=[H:K]$)), then:
 $$\mathbb T_z(\mathbb H_\infty)=\oplus _{i=1}^n e_i \mathbb T_z(\mathbb K_\infty).$$

Observe also that $G$ acts on $\mathbb T_z(\mathbb H_\infty)$ and $\mathbb T_z(\mathbb H_\infty)$ is a free $\mathbb T_z(\mathbb K_\infty)[G]$-module of rank one by the normal basis Theorem. We denote by $\mathbb  T_z(\mathbb H_\infty)[G]$ the  ring:
 $$\mathbb  T_z(\mathbb H_\infty)[G] :=\oplus_{\sigma \in G} \mathbb T_z(\mathbb H_\infty)\sigma,$$
where the product rule is given by:
 $$\forall \sigma_1, \sigma_2\in  G, \forall g_1, g_2\in \mathbb T_z(\mathbb H_\infty), \quad (g_1\sigma_1)\, (g_2\sigma_2)= g_1\sigma_1(g_2)\,  \sigma_1\sigma_2.$$
Let's set:
 $$\exp_{\widetilde{\varphi}}=\sum_{i\geq 0}\frac{f\cdots f^{(i-1)}}{f\cdots f^{(i-1)}\mid_{\xi^{(i)}}}z ^i\tau^i\in \mathbb H[z]\{\{\tau \}\}.$$
Let $I$ be a non-zero ideal  of $A.$ We set:
 $$u_I=\sum_{i=0}^{\deg I} \phi_{I,i} f\cdots f^{(i-1)}\in W(B),$$
where $\phi_I=\sum_{i=0}^{\deg I} \phi_{I,i} \tau^i, $ $\phi_{I,i} \in B.$ Note that if $I=aA,$ we have:
 $$u_I=\frac{\rho(a)}{\sgn(a)}.$$
Furthermore, we prove (see \cite{AND}, Section 3.7 for the case $d_\infty=1$):

\begin{lemma}
Let $I, J$ be two non-zero ideals of $A$. We have:
 $$u_I\mid_{\xi}=\psi(I),$$
 $$\sigma_I(f)u_I=f\tau(u_I),$$
 $$u_{IJ}= \sigma_I(u_J)u_I.$$
\end{lemma}

\begin{proof}
The fact that $u_I\mid_{\xi}=\psi(I)$ comes from the definition of $u_I.$ Note that we have a natural isomorphism of $B$-modules:
 $$\gamma_\phi: W(B)\simeq  B\{\tau\}, \quad f\cdots f^{(i-1)}\mapsto \tau^i.$$
In particular:
 $$\forall x\in W(B), \quad \gamma_\phi (fx^{(1)})= \tau \gamma_\phi (x),$$
 $$\forall x\in W(B), \forall a\in A, \quad \gamma_\phi (\rho(a)x)=\gamma_\phi(x)\phi_a.$$
By explicit reciprocity law (see \cite{GOS}, Theorem 7.4.8), we have:
 $$\forall a \in A, \quad \phi_I \phi_a=\sigma_I(\phi)_a \phi_I.$$
By direct calculations, we deduce from this:
 $$\sigma_I(f)u_I=f\tau(u_I),$$

Now, let $J$ be a non-zero ideal of $A.$ We have:
$$\gamma_\phi(u_{IJ})= \phi_{IJ}= \sigma_I(\phi_J)\phi_I.$$
But, since $\forall i\geq 0,$ $\sigma_I(f\cdots f^{(i-1)})u_I= f\cdots f^{(i-1)}u_I^{(i)},$ we have :
$$\gamma_\phi (\sigma_I(u_J) u_I)=\sigma_I(\phi_J) \phi_I.$$
Thus:
$$u_{IJ}= \sigma_I(u_J)u_I.$$
\end{proof}

We deduce that if $P, Q$ are maximal ideals of $A:$
 $$(1-\frac{u_P}{\psi(P)}z^{\deg P}\sigma_P)(1-\frac{u_Q}{\psi(Q)}z^{\deg Q}\sigma_Q)=(1-\frac{u_Q}{\psi(Q)}z^{\deg (Q)}\sigma_Q)(1-\frac{u_P}{\psi(P)}z^{\deg P}\sigma_P).$$
For every integer $n\geq 1,$ we set:
 $$(1-\frac{u_P}{\psi(P)^n}z^{\deg P}\sigma_P)^{-1}:=\sum_{k\geq 0} \frac{u_{P^k}}{\psi(P^k)^n}z^{k\deg P} \sigma_{P^k} \in \mathbb T_z (\mathbb H_\infty)[G].$$
We define:
 $$\forall n\geq 1, \quad \mathcal L(\varphi;n;z)=\prod_P(1-\frac{u_P}{\psi(P)^n}z^{\deg P}\sigma_P)^{-1}\in (\mathbb  T_z(\mathbb H_\infty)[G])^\times,$$
where $P$ runs through the maximal ideals of $A.$ Note that, for any $n\geq 1,$ $  \mathcal L(\varphi;n;z)$ induces a $\mathbb T_z(\mathbb K_\infty)$-linear endomorphism of $\mathbb T_z(\mathbb H_\infty),$ and we denote by $\det_{\mathbb T_z(\mathbb K_\infty)}\mathcal L(\varphi;n;z)$ its determinant. Let's set:
 $$W(B[z])=\oplus _{i\geq 0} B[z] f\cdots f^{(i-1)}\subset \mathbb H[z].$$

\begin{proposition}\label{PropositionS7.1}
We have:
 $$\forall n\geq 1, \quad {\rm det}_{\mathbb T_z(\mathbb K_\infty)}\mathcal L(\varphi;n;z)=\prod_{\frak P}(1-\frac{\rho([\frac{B}{\frak P B}]_A)z^{\deg N_{H/K}(\frak P)}}{[\frac{B}{\frak P B}]_A^n})^{-1}\in \mathbb T_z({\mathbb  K}_\infty)^\times,$$
where $\frak P$ runs through the maximal ideals of $B.$
\end{proposition}

\begin{proof} The proof is similar to that of Theorem \ref{TheoremS6.1} . We give a sketch of the proof for the convenience of the reader.

Let $n\geq 1.$ We have:
 $${\rm det}_{\mathbb T_z(\mathbb K_\infty)}\mathcal L(\varphi;n;z)=\prod_P{\rm det}_{\mathbb K[z]}(1-\frac{u_P}{\psi(P)^n}z^{\deg P}\sigma_P\mid_{\mathbb H[z]})^{-1}.$$
Let $P$ be a maximal ideal of $A.$ Let $e\geq 1$ be the order of $P$ in ${\rm Pic}(A).$ Then $1,\sigma_P, \ldots, \sigma_P^{e-1}$ are linearly independent over $\mathbb H(z).$  We have :
 $$(\frac{u_P}{\psi(P)^n}z^{\deg P}\sigma_P)^e=\frac{\rho(\psi(P^e)) z^{e\deg P}}{\psi (P^e)^n}\in \mathbb K[z].$$
Thus the minimal polynomial of $\frac{u_P}{\psi(P)^n}z^{\deg P}\sigma_P\mid_{\mathbb H(z)}$ over $\mathbb K(z)$ (and also over $\mathbb H^{\langle \sigma_P \rangle}(z)$) is equal to:
 $$X^e-\frac{\rho(\psi(P^e)) z^{e\deg P}}{\psi (P^e)^n}\in \mathbb K[z][X].$$
Therefore the characteristic polynomial of $\frac{u_P}{\psi(P)^n}z^{\deg P}\sigma_P\mid_{\mathbb H(z)}$ over $\mathbb K(z)$ is equal to:
 $$(X^e-\frac{\rho(\psi(P^e)) z^{e\deg P}}{\psi (P^e)^n})^{\frac{[H:K]}{e}}.$$
One obtains the desired result by the same arguments as that used in the proof of Theorem \ref{TheoremS6.1}.
\end{proof}

\begin{remark}\label{RemarkOmega}
Let $L=\rho(K)(\mathbb F_\infty)((^{q^{d_\infty}-1}\sqrt{-\pi})),$ and let $\tau: L\rightarrow L$ be the continuous morphism of $\rho(K)$-algebras such that $\forall x\in \mathbb F_\infty((^{q^{d_\infty}-1}\sqrt{-\pi})), \tau(x)=x^q.$ Then there exists an element  $\omega\in L^\times$ (unique up to the multiplication of an element in $\rho(K)^\times$) such that:
$$\tau (\omega)=f\omega.$$
This element is a generalization of the special function introduced by G. Anderson and D. Thakur  in \cite{AND&THA}. The existence of this element (combined with the log-algebraicity theorem) gives new arithmetic informations on special values of $L$-series.  We refer the interested reader to a forthcoming work of the authors. 
\end{remark}

%%%%%%%%%%%%%%%%%%%%%%%%%%%%%

\subsection{Stark units and several variable log-algebraicity theorem}${}$

\medskip

We set:
 $$U(\widetilde{\varphi}/W(B[z]))=\{ x\in \mathbb T_z(\mathbb H_\infty), \exp_{\widetilde{\varphi}}(x)\in W(B[z])\}.$$
The following result is a twisted (by the shtuka function $f$) version of \cite{AND}, Theorem 5.1.1. :

\begin{theorem}\label{TheoremS7.1}
We have:
 $$U(\widetilde{\varphi}/W(B[z]))=\mathcal L(\varphi;1;z) W(B[z]).$$
In particular,
 $$\exp_{\widetilde{\varphi}}(\mathcal L(\varphi;1;z) W(B[z]))\subset W(B[z]),$$
\end{theorem}

\begin{proof}
The proof is similar to that of Theorem \ref{TheoremS6.2}. We give a sketch of the proof for the convenience of the reader.\par
Observe that $\exp_{\widetilde{\varphi}}:\mathbb H[[z]]\rightarrow \mathbb H[[z]]$ is an isomorphism of $\mathbb A[[z]]$-modules. Furthermore, if we set:
 $$W(H[[z]])=\oplus_{i\geq 0}H[[z]] f\cdots f^{(i-1)},$$
we get:
 $$\exp_{\widetilde{\varphi}}(W(H[[z]]))=W(H[[z]]).$$
Let:
 $$W(B[[z]])=\oplus _{i\geq 0} B[[z]]f\cdots f^{(i-1)}\subset \mathbb H[[z]].$$
Let $P$ be a maximal ideal of $A.$ Let $W_P=S^{-1}W(B[[z]]),$ where $S=A\setminus P.$ Then:
 $$PW_P=\psi(P)W_P.$$
By Lemma \ref{LemmaS6.5}, we have:
 $$\exp_{\widetilde{\varphi}}(PW_P)=PW_P.$$
If $$\phi_P= \sum_{i=0}^{\deg P} \phi_{P,i} \tau ^i,$$
we set:
 $$\widetilde{\varphi}_P= \sum_{i=0}^{\deg P} \phi_{P,i} f\cdots f^{(i-1)}z^i \tau^i.$$
We have:
 $$\widetilde{\varphi}_P\exp_{\widetilde{\varphi}} = \exp_{\sigma_P}\, \psi(P),$$
where:
 $$\exp_{\sigma_P}=\sum_{i\geq 0} \sigma_P(e_i(\phi)) f\cdots f^{(i-1)} z^i \tau^i.$$
Let's set:
 $$U(\widetilde{\varphi}/W_P)=\{ x\in \mathbb H[[z]], \exp_{\widetilde{\phi}}(x)\in W_P\}\subset W(H[[z]]).$$
We have  an isomorphism of $A[[z]]$-modules induced by $\exp_{\widetilde{\varphi}}$:
 $$\frac{U(\widetilde{\varphi}/W_P)}{PW_P}\simeq \widetilde{\varphi}(\frac{W_P}{PW_P}).$$
Note that:
 $$\forall i\geq 0, \quad \sigma_P (f\cdots f^{(i-1)}) u_P = f\cdots f^{(i-1)}\tau^i (u_P)\in W(B).$$
Therefore:
 $$(\widetilde{\varphi}_P-z^{\deg P}u_P \sigma_P)\widetilde{\varphi}(\frac{W_{P}}{PW_{P}})=\{0\}.$$
Since $u_P$ is a ``$P$-unit'', for $x\in W( H[[z]])\setminus W_P,$ $(\widetilde{\varphi}_P-z^{\deg P} u_P\sigma_P)(x)$ is not $P$-integral as an element of $\mathbb H[[z]].$
Thus:
 $$\widetilde{\varphi}(\frac{W_P}{PW_P[[z]]})=\{ x\in \widetilde{\varphi}(\frac{W( H[[z]])}{PW_P}),(\widetilde{\varphi}_P-z^{\deg P} u_P\sigma_P)(x)=0\}.$$
Let $x\in W(H[[z]]),$ we deduce that:
 $$x\in U(\widetilde{\varphi}/W_P)\Leftrightarrow (\widetilde{\varphi}_P-z^{\deg P} u_P\sigma_P)(\exp_{\widetilde{\varphi}}(x))\in PW_P.$$
Thus:
 $$x\in U(\widetilde{\varphi}/W_P)\Leftrightarrow \exp_{\sigma_P} (\psi(P) x-z^{\deg P} u_P\sigma_P(x))\in PW_P.$$

Lemma \ref{LemmaS6.5} implies:
 $$x\in U(\widetilde{\varphi}/W_P)\Leftrightarrow \psi(P)x-z^{\deg P}u_P \sigma_P(x)\in PW_P.$$
Thus:
 $$U(\widetilde{\varphi}/W_P)= (1-\frac{z^{\deg P}u_P}{\psi(P)}\sigma_P)^{-1} W_P.$$
Observe that $W(B[[z]])=\bigcap_PW_P.$ We conclude that:
 $$W(B[[z]])=\exp_{\widetilde{\varphi}}(\mathcal L(\varphi;1;z)W(B[[z]])).$$
By Lemma \ref{LemmaS7.1}, we get:
 $$\exp_{\widetilde{\varphi}}(\mathcal L(\varphi;1;z)W(B[z]))\subset  \mathbb T_z(\mathbb H_\infty)\cap W(B[[z]])= W(B[z]).$$
Recall that:
 $$U(\widetilde{\varphi}/W(B[z]))=\{ x\in \mathbb T_z(\mathbb H_\infty), \exp_{\widetilde{\varphi}}(x)\in W(B[z])\}.$$
Then:
 $$U(\widetilde{\varphi}/W(B[z]))=\mathcal L(\varphi;1;z)W(B[[z]]) \cap \mathbb T_z(\mathbb H_\infty).$$
But recall that:
 $$\mathcal L(\varphi;1;z)\in (\mathbb T_z(\mathbb H_\infty)[G])^\times.$$
Thus:
 $$U(\widetilde{\varphi}/W(B[z]))=\mathcal L(\varphi;1;z)W(B[z]).$$
\end{proof}

Let $\ev:\mathbb T_z(\mathbb H_\infty)\rightarrow \mathbb H_\infty$ be the evaluation map at $z=1.$ Then by Proposition \ref{PropositionS7.1}, we get:
 $$L(\varphi/W(B))={\rm det}_{\mathbb K_\infty}  \ev(\mathcal L(\varphi;1;z)),$$
where:
 $$\ev(\mathcal L(\varphi;1;z))=\prod_P(1-\frac{u_P}{\psi(P)}\sigma_P)^{-1}=\sum_{I} \frac{u_I}{\psi (I)} \sigma_I\in (\mathbb H_\infty[G])^\times,$$
where $I$ runs through the non-zero ideals of $A.$ Furthermore, by the above Theorem:
 $$\exp_\varphi(\ev(\mathcal L(\varphi;1;z))W(B))\subset W(B).$$
And also:
 $$\exp_\varphi(\ev(\mathcal L(\varphi;1;z))W(B)\mathbb B)\subset W(B)\mathbb B.$$
If we define the regulator of Stark units $\ev(\mathcal L(\varphi;1;z))W(B)\mathbb B$ as follows:
 $$[W(B)\mathbb B:\ev(\mathcal L(\varphi;1;z))W(B)\mathbb B]_{\mathbb A}:= {\rm det}_{ \mathbb K_\infty}\ev(\mathcal L(\varphi;1;z)),$$
then:
 $$L(\varphi/W(B))=[W(B)\mathbb B:\ev(\mathcal L(\varphi;1;z))W(B)\mathbb B]_{\mathbb A}.$$\par
We now briefly discuss the several  variable version of Theorem \ref{TheoremS7.1}. Let $s\geq 0$ be an integer. Let:
 $$K_s= {\rm Frac} (A^{\otimes s}),$$
where:
 $$A^{\otimes s}=A\otimes_{\mathbb F_q}\cdots \otimes _{\mathbb F_q}A.$$
If $s=0,$ then $K_0=\mathbb F_q.$ Let:
 $$\mathbb H_s={\rm Frac}(A^{\otimes s} \otimes_{\mathbb F_q} B),$$
 $$\mathbb K_s={\rm Frac}(A^{\otimes s} \otimes_{\mathbb F_q} A).$$
For $i=1, \ldots, s,$ let:
 $$\rho_i : A\rightarrow \mathbb H_s, \quad a\mapsto (1\otimes \cdots 1\otimes a\otimes \cdots \otimes 1) \otimes 1,$$
where $a$ appears at the $i$-th position. We still denote by $\rho_i: \mathbb H\rightarrow \mathbb H_s$ the homomorphism of $H$-algebras such that:
 $$\forall a\in A, \quad \rho_i(\rho(a))=\rho_i(a).$$
We view $\mathbb H_s$ and $\mathbb K_s$ as functions fields over $K_s \otimes 1.$ Let $\infty$ be the unique place of $\mathbb K_s/K_s \otimes 1$ above the place $\infty$ of $K/\mathbb F_q.$ Then:
 $$\mathbb K_{s, \infty} =(K_s \otimes 1)(\mathbb F_\infty)((1^{\otimes s} \otimes \pi)),$$
and we set:
 $$\mathbb H_{s, \infty}= \mathbb H_s\otimes_{\mathbb K_s}\mathbb K_{s, \infty}.$$
Let $\mathbb T_z(\mathbb H_{s, \infty})$ be the Tate algebra in the variable $z$ with coefficients in $\mathbb H_{s, \infty}.$
Let $\tau:\mathbb T_z(\mathbb H_{s, \infty})\rightarrow\mathbb T_z(\mathbb H_{s, \infty})$ be the continuous homomorphism of $(K_s \otimes 1)[z] $-algebras such that:
 $$\forall x\in H_\infty, \quad \tau (x)=x^q.$$
Let's set:
 $$W_s(B[z])=\oplus_{i_1, \ldots i_s\geq 0} B[z] \prod_{j=1}^s \rho_j(f)\cdots \tau^{(i_j-1)}(\rho_j(f))\subset \mathbb H_s[z].$$
In particular $W_0(B[z])=B[z].$
By similar arguments as those of the proof of Lemma \ref{LemmaS7.1}, we show that $W_s(B[z])$ is discrete in $\mathbb T_z(\mathbb H_{s, \infty}).$
For $n\geq 1,$  we set:
 $$\mathcal L(\varphi_s;n;z):=\prod_P(1-\frac{\prod_{j=1}^s \rho_j(u_P)}{\psi(P)^n}z^{\deg P}\sigma_P)^{-1}\in (\mathbb  T_z(\mathbb H_{s,\infty})[G])^\times,$$
where $P$ runs through the maximal ideals of $A.$ Then, by the same proof as that of Proposition \ref{PropositionS7.1}, for all $n\geq 1$, we get:
 $${\rm det}_{\mathbb T_z(\mathbb K_{s,\infty})}\mathcal L(\varphi_s;n;z)=\prod_{\frak P}(1-\frac{(\prod_{j=1}^s \rho_j([\frac{B}{\frak P B}]_A))z^{\deg N_{H/K}(\frak P)}}{[\frac{B}{\frak P B}]_A^n})^{-1}\in \mathbb T_z({\mathbb  K}_{s,\infty})^\times,$$
where $\frak P$ runs through the maximal ideals of $B.$

\medskip

We define:
 $$\exp_{\widetilde{\varphi}_s}=\sum_{i\geq 0} e_i(\phi) (\prod_{j=1}^s \rho_j(f)\cdots \tau^{i-1}(\rho_j(f)))z^i\tau ^i\in \mathbb H_s\{\{\tau\}\}.$$
Then $\exp_{\widetilde{\varphi}_s}$ converges on $\mathbb T_s(\mathbb H_{s, \infty}),$ and we set:
 $$U(\widetilde{\varphi}_s/W_s(B[z]))=\{ x\in \mathbb T_z(\mathbb H_{s,\infty}), \exp_{\widetilde{\varphi}_s}(x)\in W_s(B[z])\}.$$
By similar arguments as those of the proof of Theorem \ref{TheoremS7.1}, we get:

\begin{corollary}\label{CorollaryS7.1}
We have:
 $$U(\widetilde{\varphi}_s/W_s(B[z]))= \mathcal L(\varphi_s;1;z)W_s(B[z]).$$
\end{corollary}

%%%%%%%%%%%%%%%%%%%%%%%%%%%%%%%%%%%%%%%%%%%%%%%%

\begin{example}\label{ExampleS7.1}${}$
We consider the Carlitz example, where $g=0$ and $d_\infty=1.$  Observe that there exists $\theta\in K$ such that $\sgn(\theta)=1,$  and $A=\mathbb F_q[\theta].$ Thus, $K=\mathbb F_q(\theta),$ and $K_\infty=\mathbb F_q((\frac{1}{\theta})).$

Let $\phi:A \rightarrow \overline{K}_\infty\{\tau\}$ be the Carlitz module defined by
 $$\phi_\theta= \theta+\tau.$$
Then the Carlitz exponential is given by:
 $$\exp_\phi=\sum_{i\geq 0}\frac{1}{D_i} \tau^i,$$
where for $i\geq 0,$ $D_i=\prod_{k=0}^{i-1} (\theta^{q^i}-\theta^{q^k}).$

The Hilbert class field $H$ of $K$ is $K,$ and then $B=A.$ Then, the shtuka function $f \in K \otimes H$ associated to the Carlitz module via the Drinfeld correspondence is given by:
 $$f=\theta\otimes 1-1\otimes \theta.$$

Let $s\geq 0$ be an integer. For $i=1, \ldots, s$, let $t_i=\rho_i(\theta).$
We have:
 $$\rho_i(f)= t_i-\theta,$$
 $$\mathbb H_s=\mathbb K_s=\mathbb F_q(t_1, \ldots, t_s, \theta),$$
 $$\mathbb H_{s, \infty}=\mathbb K_{s, \infty}= \mathbb F_q(t_1, \ldots, t_s)((\frac{1}{\theta})).$$
For $i\geq 0, j=1, \cdots s,$ set:
 $$b_i(t_j)=\prod_{k=0}^{i-1}(t_j-\theta^{q^k}).$$
We get:
 $$W_s(B[z])=A[t_1, \ldots, t_s][z].$$
Observe that:
 $$\exp_{\widetilde{\varphi}_s}=\sum_{i\geq 0}\frac{\prod_{j=1}^s b_i(t_j)}{D_i} \tau^i.$$
We have:
 $$\mathcal L(\varphi_s;1;z)=\sum_{a\in A_+} \frac{a(t_1)\cdots a(t_s)}{a} z^{\deg_\theta a},$$
where $A_+$ denotes the set of monic polynomials in $A=\mathbb F_q[\theta].$ In particular, for $s=1$, we recover the zeta function introduced by Pellarin \cite{PEL}.

Corollary \ref{CorollaryS7.1} implies:
 $$\exp_{\widetilde{\varphi}_s}( \mathcal L(\varphi_s;1;z)A[t_1, \ldots, t_s,z])\subset A[t_1, \ldots, t_s,z].$$
We refer the interested reader to  \cite{ANDTR2}, \cite{APT},  \cite{APTR}, \cite{ATR}, for arithmetic applications of this latter result.
\end{example}

%%%%%%%%%%%%%%%%%%%%%%%%%%%%%%%%%

\subsection{Another proof of Anderson's log-algebraicity theorem}\label{Section another proof}${}$

\medskip

\begin{corollary}\label{CorollaryS7.2}
Let $n\geq 0$ and let $X_1, \ldots, X_n,z$ be $n+1$ indeterminates over $K.$ Let $\tau: K[X_1, \ldots, X_n][[z]]\rightarrow K[X_1, \ldots, X_n][[z]]$ be the continuous $\mathbb F_q[[z]]$-algebra homomorphism for the $z$-adic topology such that $\forall x\in K[X_1, \ldots, X_n], \tau (x)=x^q.$ Then:
 $$\forall b\in B, \quad \exp_{\widetilde{\phi}}(\sum_I \frac{\sigma_I(b)}{\psi(I)} \phi_I(X_1)\cdots \phi_I(X_n) z^{\deg I})\in B[X_1, \ldots, X_n, z],$$
where $I$ runs through the non-zero ideals of $A,$ and:
 $$\exp_{\widetilde{\phi}}=\sum_{i\geq 0} e_i(\phi) z^i \tau^i.$$
\end{corollary}

\begin{proof}
We first treat the case $n=0.$ Let $b \in B.$ By Theorem \ref{TheoremS7.1}, we get:
$$\forall k\geq 0, \quad \sum_{\deg I +i= k} e_i(\phi)f\cdots f^{(i-1)} \frac{\tau^i(u_I \sigma_I(b))}{\psi(I)^{q^i}}\in W(B),$$
and:
$$\forall k \gg 0, \quad \sum_{\deg I +i= k} e_i(\phi)f\cdots f^{(i-1)} \frac{\tau^i(u_I \sigma_I(b))}{\psi(I)^{q^i}}=0.$$
The coefficient of $f\cdots f^{(k-1)}$ in $\sum_{\deg I +i= k} e_i(\phi)f\cdots f^{(i-1)} \frac{\tau^i(u_I \sigma_I(b))}{\psi(I)^{q^i}}$ is:
 $$  \sum_{\deg I +i= k} e_i(\phi) \frac{\sigma_I(b)^{q ^i}}{\psi(I)^{q^i}}.$$
Therefore:
$$\forall k\geq 0,  \quad \sum_{\deg I +i= k} e_i(\phi) \frac{\sigma_I(b)^{q ^i}}{\psi(I)^{q^i}}\in B.$$
 $$\forall k \gg 0, \quad \sum_{\deg I +i= k} e_i(\phi) \frac{\sigma_I(b)^{q ^i}}{\psi(I)^{q^i}}=0.$$
 Thus:
 $$\exp_{\widetilde{\phi}}(\sum_I \frac{\sigma_I(b)}{\psi(I)}z^{\deg I})\in B[z].$$\par
We now assume that $n\geq 1.$  We have an isomorphism of $B[z]$-modules
 $$\gamma: W(B[z])  \rightarrow \oplus_{i_1, \ldots, i_n\geq 0}B[z] X_1^{q^{i_1}} \cdots X_n^{q^{i_n}}$$
such that:
 $$\forall i_1, \ldots, i_n \in \mathbb N, \quad \gamma( \prod_{j=1}^n \rho_j(f\cdots f^{(i_j-1)}))=\prod_{j=1}^n X_j^{q^{i_j}}.$$
 Observe that:
 $$\gamma \circ \prod_{j=1}^n \rho_j(f)\tau = \tau \circ \gamma.$$
 Furthermore:
 $$\gamma ((\prod_{j=1}^n \rho_j(u_I)))=\phi_I(X_1)\cdots \phi_I(X_n).$$
 Thus, we get by Corollary \ref{CorollaryS7.1}:
 $$\exp_{\widetilde{\varphi}_n}(\mathcal L (\varphi_n;1;z)b)\in W_n(B[z]),$$
 and thus:
 $$\exp_{\widetilde{\phi}}(\sum_I \frac{\sigma_I(b)}{\psi(I)} \phi_I(X_1)\cdots \phi_I(X_n) z^{\deg I})\in \oplus_{i_1, \ldots i_n \geq 0}B[z]X_1^{q^{i_1}} \cdots X_n^{q^{i_n}}.$$
\end{proof}

\begin{remark}\label{RemakS7.1}
Let $s\geq 1$ be an integer and let $B\{\tau_1, \ldots, \tau_s\}$ be the non-commutative polynomial ring in the variables $\tau_1, \ldots, \tau_s,$ such that:
 $$\tau_i\tau_j=\tau_j \tau_i,$$
 $$\forall b\in B, \forall n\geq 0, \quad \tau_i^n b= b^{q^n  }\tau_i.$$
For $i=1, \ldots, s,$ we set:
 $$\forall a\in A, \quad \varphi_{i,a}=\sum_{j=0}^{\deg a} \phi_{a, j}  \tau_i^j\in B\{\tau_1, \ldots, \tau_s\},$$
and:
 $$\forall a\in A, \quad \varphi_a=\sum_{j=0}^{\deg  a} \phi_{a, j} \tau^j\in B\{\tau_1, \ldots, \tau_s\},$$
where $\tau =\tau_1\cdots \tau_s.$

Let $W_s(B)=\oplus_{i_1, \ldots, i_s} B\prod_{j=1}^s \rho_j(f)\cdots \tau^{i_j-1}(\rho_j(f))\subset \mathbb H_s.$ Then $W_s(B)$ is an $A^{\otimes s} \otimes B$-module. Let $j\in \{1, \ldots s\}.$ Let $a\in A,$ we have a natural $B$-module homomorphism:
 $$\widetilde{\rho}_j(a):  B\{\tau_1, \ldots, \tau_s\}\rightarrow B\{\tau_1, \ldots, \tau_s\},$$
such that:
 $$\forall i_1,\ldots, i_s\in \mathbb N, \quad \widetilde{\rho}_j(a).( \tau_1^{i_1} \cdots \tau_s^{i_s})= (\tau_j^{i_j}\varphi_{j, a})\prod_{k=1, k\not =j}^s\tau_k^{i_k}.$$
Observe that:
 $$\forall i,j \in \{1, \ldots s\}, \forall a,b\in A, \quad \widetilde{\rho}_j(a)\circ \widetilde{\rho}_i(b)=\widetilde{\rho}_i(b)\circ \widetilde{\rho}_j(a).$$
Thus  $B\{\tau_1, \ldots, \tau_s\}$ becomes an $A^{\otimes s} \otimes B$-module via:
 $$\forall x\in B\{\tau_1, \ldots, \tau_s\} , \quad (\sum_ib_i\prod_{j=1}^s \rho_j(a_{i,j})) \cdot x=\sum_ib_i(\prod_{j=1}^s \widetilde{\rho}_j(a_{i,j}))(x).$$
Then, by the proof of Corollary \ref{CorollaryS7.2}, we have an $A^{\otimes s} \otimes B$-module isomorphism:
 $$B\{\tau_1, \ldots, \tau_s\}\simeq W_s(B).$$
In particular, $B\{\tau_1, \ldots, \tau_s\}$ is a finitely generated $A^{\otimes s} \otimes B$-module of rank one. The case $s=1$ was already observed by G. Anderson (\cite{GOS}, page 230, line 21 - there is a misprint in line 24, since in general $f \not \in A \otimes _{\mathbb F_q}\mathbb C_\infty$). If $I$ is a non-zero ideal of $A,$ we define  $I*\cdot:B\{\tau_1, \ldots, \tau_s\}\rightarrow B\{\tau_1, \ldots, \tau_s\}$ to be the $B$-module homomorphism such that:
 $$I*(\tau_1^{i_1} \cdots \tau_s^{i_s})= \sum_{j_1, \ldots, j_s\in \{0, \ldots, \deg I\}} \phi_{I,j_1}^{q^{i_1}}\cdots \phi_{I,j_s}^{q^{i_s}}\tau_1^{i_1+j_1}\cdots \tau_s^{i_s+j_s},$$
where $\phi_I= \sum_{k=0}^{\deg I} \phi_{I,k} \tau^k.$

Let $\mathcal L: B\{\tau_1, \ldots, \tau_s\}\rightarrow H\{\{\tau_1, \ldots, \tau_s\}\}$ be defined as follows:
$$\mathcal L (\sum_{i_1, \ldots i_s} b_{i_1, \ldots, i_s} \tau_1^{i_1} \cdots \tau_s ^{i_s})= \sum_{i_1, \ldots, i_s}\sum_I \frac{\sigma_I(b_{i_1, \ldots, i_s})}{\psi(I)} I*(\tau_1^{i_1}\cdots \tau_s^{i_s}).$$
Then by Corollary \ref{CorollaryS7.1}, we get that the multiplication by $\exp_\phi\in H\{\{\tau \}\}$ on $H\{\{\tau_1, \ldots, \tau_s\}\}$ yields to the following property:
$$\forall x\in B\{\tau_1, \ldots, \tau_s\}, \quad \exp_\phi (\mathcal L(x)) \in B\{\tau_1, \ldots, \tau_s\}.$$
\end{remark}

%%%%%%%%%%%%%%%%%%%%%%%%%%%%%

\end{document}